\documentclass[12pt]{amsart}
\usepackage[margin=1.2in]{geometry}
\usepackage{graphicx,latexsym}
\usepackage{tikz}
\usepackage{amsfonts, amssymb, amsmath, amsthm}

\newcommand{\N}{\mathbb{N}}
\newcommand{\Z}{\mathbb{Z}}
\newcommand{\R}{\mathbb{R}}
\newcommand{\C}{\mathbb{C}}

\newcommand{\zw}{\mathcal{A}_{\omega}}

\newcommand{\ow}{\mathcal{O}_{\omega}}
\newcommand{\of}{\mathcal{O}_{(\omega)}}

\newcommand{\pw}{\mathcal{P}_{\omega}}

\newcommand{\dx}{{\rm d}x }

\newcommand{\dv}{{\rm d}v }
\newcommand{\dt}{{\rm d}t }
\newcommand{\dxi}{{\rm d}\xi }
\newcommand{\dz}{{\rm d}z }
\newcommand{\dml}{{\rm d}m(\lambda) }
\newcommand{\dl}{{\rm d}\lambda }
\newcommand{\dzeta}{{\rm d}\zeta }

\usetikzlibrary{arrows,chains,matrix,positioning,scopes}
\makeatletter
\tikzset{join/.code=\tikzset{after node path={%
\ifx\tikzchainprevious\pgfutil@empty\else(\tikzchainprevious)%
edge[every join]#1(\tikzchaincurrent)\fi}}}
\makeatother
\tikzset{>=stealth',every on chain/.append style={join},
         every join/.style={->}}
\tikzstyle{labeled}=[execute at begin node=$\scriptstyle,
   execute at end node=$]

\newtheorem{theorem}{Theorem}[section]
\newtheorem{proposition}[theorem]{Proposition}
\newtheorem{lemma}[theorem]{Lemma}
\newtheorem{corollary}[theorem]{Corollary}

\theoremstyle{definition}

\theoremstyle{remark}
\newtheorem{remark}[theorem]{Remark}
\newtheorem{examples} [theorem]{Examples}

\numberwithin{equation}{section}

\begin{document}
\title[Hyperfunctions and ultrahyperfunctions of fast growth]{On the non-triviality of certain spaces of analytic functions. Hyperfunctions and ultrahyperfunctions of fast growth}

\author[A. Debrouwere]{Andreas Debrouwere}
\address{Department of Mathematics, Ghent University, Krijgslaan 281, 9000 Gent, Belgium}
\email{Andreas.Debrouwere@UGent.be}
\thanks{A. Debrouwere gratefully acknowledges support by Ghent University, through a BOF Ph.D.-grant.}

\author[J. Vindas]{Jasson Vindas}
\thanks{The work of J. Vindas was supported by Ghent University, through the BOF-grant 01N01014.}
\address{Department of Mathematics, Ghent University, Krijgslaan 281, 9000 Gent, Belgium}
\email{Jasson.Vindas@UGent.be}

\begin{abstract}
We study function spaces consisting of analytic functions with fast decay on horizontal strips of the complex plane with respect to a given weight function. 
 Their duals, so called spaces of (ultra)hyperfunctions of fast growth, generalize the spaces of Fourier hyperfunctions and Fourier ultrahyperfunctions.  An analytic representation theory for their duals is developed and applied to characterize the non-triviality of these function spaces 
 in terms of the growth order of the weight function. In particular, we show that the Gelfand-Shilov spaces of Beurling type $\mathcal{S}^{(p!)}_{(M_p)}$ and Roumieu type $\mathcal{S}^{\{p!\}}_{\{M_p\}}$ are non-trivial if and only if
$$
\sup_{p \geq 2}\frac{(\log p)^p}{h^pM_p} < \infty,
$$
for all $h > 0$ and some $h > 0$, respectively. We also study boundary values of holomorphic functions in spaces of ultradistributions of exponential type, which may be of quasianalytic type.

\end{abstract}

\maketitle

%
%
%
\section{Introduction}

The purpose of this paper is to introduce and analyze two families of spaces of analytic functions and provide an analytic representation theory for their duals. Our function spaces consist of analytic functions with very fast decay on strips of the complex plane with respect to a weight function. Their duals lead to new classes of hyperfunctions and ultrahyperfunctions of `fast growth', and generalize the Fourier hyperfunctions and the Fourier ultrahyperfunctions.

Fourier hyperfunctions were systematically studied by Kawai in \cite{Kawai}, and their local theory includes that of Sato's hyperfunctions \cite{Kaneko,Morimoto2}. In one dimension, their global sections on $\overline{\mathbb{R}}=[-\infty,\infty]$ are the dual of the space of analytic functions with exponential decay on some (horizontal) strip, the latter test function space coincides with the Gelfand-Shilov space $\mathcal{S}^{\{1\}}_{\{1\}}=\mathcal{S}^{\{p!\}}_{\{p!\}}$ of Roumieu type \cite{PilipovicK,Gelfand}. Moreover, in the cohomological approach, this dual space can be represented as the quotient of the space of analytic functions defined outside the real line and having infra-exponential growth outside every strip containing the real line modulo its subspace of entire functions of infra-exponential type.  We mention that the use of analytic functions for the representation of dual spaces has a long tradition, which goes back to the pioneer works of K\"{o}the \cite{Kothe2,Kothe3} and Silva \cite{Silva1956,Silva}. We refer to the monographs \cite{PilipovicK,C-M} for accounts on analytic representations of (ultra)distributions, see also \cite{d-p-v,f-g-g} for recent results.

Interestingly, several basic problems in the theory of PDE naturally lead to ultradistribution spaces whose elements are not hyperfunctions \cite{Gelfand2}. Important instances of such spaces are the spaces of tempered ultrahyperfunctions and Fourier ultrahyperfunctions, introduced in one-dimension by Silva \cite{Silva} and in several variables by Hasumi \cite{hasumi} and Park and Morimoto \cite{Park}, see also \cite{Hoskins,Morimoto,Morimoto-78,SousaPinto,zharinov}. 
More recently, microlocal analysis, edge of the wedge theorems, and Bochner-Schwartz theorems in the context of ultrahyperfunctions have been investigated in \cite{b-n2010,franco-renoldi2007,y-s2006}; applications of tempered ultrahyperfunctions can be found e.g. in \cite{estrada-vindasB2013,oka-yoshino}. Also in recent times, ultrahyperfunctions have shown to be quite useful in mathematical physics, particularly as a framework for Wightman-type axiomatic formulations of relativistic quantum field theory with a fundamental length \cite{b-n2004,franco-l-r,N-B,soloviev}.

In this article we are interested in the following generalization of Kawai's and Silva's works (see also \cite[Chap.\ 1]{Gelfand2}).
Let $\omega:[0,\infty)\to[0,\infty)$ be a non-decreasing function. For $h>0$ we denote by $T^{h}$ the horizontal strip of the complex plane $|\operatorname{Im}  z|<h$. We shall study the space $\mathcal{U}_{(\omega)}(\C)$ of entire functions $\varphi$ satisfying
\begin{equation}\label{eq1intro}
 \sup_{z \in T^h}|\varphi(z)|e^{\omega(\lambda |\operatorname{Re}z|)} < \infty, 
 \end{equation}
for every $h,\lambda>0$, and the space $\mathcal{A}_{\{\omega\}}(\R)$ of analytic functions $\varphi$ defined on some strip $T^h$and satisfying the estimate (\ref{eq1intro})
for some $h, \lambda > 0$. We call their duals $\mathcal{U}'_{(\omega)}(\C)$ and  $\mathcal{A}'_{\{\omega\}}(\R)$ the spaces of \emph{ultrahyperfunctions of $(\omega)$-type} and \emph{hyperfunctions of $\{\omega\}$-type}, respectively. When $\omega(t)=t$, one recovers the spaces of Fourier ultrahyperfunctions and Fourier hyperfunctions.

The first natural question to be addressed is whether these spaces are non-trivial. We shall provide a necessary and sufficient condition on the growth of $\omega$ characterizing the non-triviality of both spaces. One of our main results asserts that $\mathcal{U}_{(\omega)}(\C)$ contains a non-identically zero function if and only if 

\begin{equation}
\label{eq2intro}\lim_{t\to\infty} e^{-\mu t}\omega(t)=0,
\end{equation}
for each $\mu>0$, while the corresponding non-triviality assertion holds for $\mathcal{A}_{\{\omega\}}(\R)$ if only if (\ref{eq2intro}) is satisfied for some $\mu>0$. We remark that this characterization is of similar nature to the Denjoy-Carleman theorem in the theory of ultradifferentiable functions \cite{Bjorck,Komatsu}. In the case of $\mathcal{A}_{\{\omega\}}(\R)$, the result will follow from complex analysis arguments. The analysis of $\mathcal{U}_{(\omega)}(\C)$ requires a more elaborate treatment, involving duality arguments and analytic representations.

It is worth pointing out that when $\omega$ is the associated function \cite{Komatsu} of a logarithmically convex weight sequence  $(M_{p})_{p\in\mathbb{N}}$, our test function spaces coincide with Gelfand-Shilov spaces of mixed type \cite{PilipovicK,Gelfand}, that is, $\mathcal{U}_{(\omega)}(\C)=\mathcal{S}^{(p!)}_{(M_p)}$  and $\mathcal{A}_{\{\omega\}}(\R)=\mathcal{S}^{\{p!\}}_{\{M_p\}}$. Specializing our result, we obtain (cf. Proposition \ref{pro-epsilon-sequences}): $\mathcal{S}^{(p!)}_{(M_p)}$ and $\mathcal{S}^{\{p!\}}_{\{M_p\}}$ are non-trivial if and only if the weight sequence satisfies the mild lower bound
\begin{equation}\label{eq3intro}
\sup_{p \geq 2}\frac{(\log p)^p}{h^pM_p} < \infty,
\end{equation}
for all $h > 0$ in the Beurling case and for some $h > 0$ in the Roumieu case; see Example \ref{ex1sequences} for instances of sequences that satisfy (\ref{eq3intro}).
Notice that finding precise conditions on two weight sequences that characterize the non-triviality of $\mathcal{S}^{(N_p)}_{(M_p)}$ and $\mathcal{S}^{\{N_p\}}_{\{M_p\}}$ is a long-standing open question, raised by Gelfand and Shilov \cite[Chap.\ 1]{Gelfand2}. Our result then solves this question when one fixes $N_{p}=p!$.

 Our second goal is to give an analytic representation theory for the dual spaces $\mathcal{U}'_{(\omega)}(\C)$ and $\mathcal{A}'_{\{\omega\}}(\R)$. We will show that every ultrahyperfunction of $(\omega)$-type (hyperfunction of $\{\omega\}$-type) can be represented as the boundary value of an analytic function defined outside some strip (outside the real line) and satisfying bounds of type $O(e^{\omega(\lambda |\operatorname*{Re} z|)})$ on substrips of its domain. Furthermore, we prove a result concerning the analytic continuation of functions whose boundary values give rise to the zero functional, which can either be viewed as a weighted version of Painlev\'e's theorem on analytic continuation or as a one-dimensional version of the edge of the wedge theorem. This allows us to express $\mathcal{U}'_{(\omega)}(\C)$ and $\mathcal{A}'_{\{\omega\}}(\R)$ as quotients of certain spaces of analytic functions. We mention that in order to establish the non-triviality of $\mathcal{U}_{(\omega)}(\C)$, we should already pass through the analytic representation theory of some intermediate duals of certain spaces of analytic functions.

As an application of our ideas, we shall study in the last part of the article boundary values of holomorphic functions in spaces of ultradistributions of exponential type. Such spaces are defined in terms of a weight function $\omega$ that is subadditive but not necessarily non-quasianalytic, and they are the duals of spaces of ultradifferentiable functions that generalize the Hasumi-Silva space $\mathcal{K}_{1}(\mathbb{R})$ of exponentially rapidly decreasing smooth functions \cite{Hoskins,hasumi,zielezny}. We use here some variants of the theory of almost analytic extensions \cite{PilipovicK,Petzsche} and Laplace transform characterizations of several interesting subspaces of Fourier (ultra)hyperfunctions. In order to study the Laplace transform, we introduce a notion of support for (ultra)hyperfunctions of fast growth in the spirit of Silva \cite{Silva} and provide a support separation theorem.

The paper is organized as follows. Section \ref{sect-weight} is devoted to the study of some properties of weight functions. Many crucial arguments in the article depend upon the existence of analytic functions satisfying certain lower and upper bounds with respect to a  weight function on a strip, Section \ref{harmonic} deals with the construction of such analytic functions. We also prove there a quantified Phragm\'en-Lindel\"of type result for analytic functions on strips. Basic properties of the test function spaces $\mathcal{U}_{(\omega)}(\C)$ and $\mathcal{A}_{\{\omega\}}(\R)$ are discussed in Section \ref{testfunctions}, where we determine their images under the Fourier transform as well. The non-triviality theorem for $\mathcal{A}_{\{\omega\}}(\R)$ is also shown in Section \ref{testfunctions}. In Section \ref{section boundary values} we present the analytic representation theory for $\mathcal{U}'_{(\omega)}(\C)$ and $\mathcal{A}'_{\{\omega\}}(\R)$, and, as an application, we characterize the non-triviality of the space $\mathcal{U}_{(\omega)}(\C)$. We introduce the notion of (real) support in Section \ref{sect-support} and provide a support separation theorem. In Section \ref{section subadditive}, we give a variant of the theory from the previous sections that applies to spaces defined via subadditive weights. For a weight function $\omega$, the modification consists in replacing (\ref{eq1intro}) in the definition of the test functions spaces by estimates of the form
\[
 \sup_{z \in T^h}|\varphi(z)|e^{\lambda \omega(|\operatorname{Re}z|)} < \infty.
\]
We mention that, in the Roumieu case, these test function spaces were investigated by Langenbruch \cite{Langenbruch} under a mild condition (much weaker than subadditivity) on the weight function $\omega$ (see Remark \ref{remark-non-triv}). If $\omega$ is subadditive, the resulting dual spaces are subspaces of the Fourier (ultra)hyperfunctions, which we employ to introduce spaces of ultradistributions of exponential type as their images under the Fourier transform, in analogy to the Beurling-Bj\"{o}rck approach to ultradistribution theory \cite{Bjorck}. We conclude the article with the study of boundary values of analytic functions in these ultradistribution spaces of exponential type in Section \ref{section laplace}. 

\section{Weight functions}\label{sect-weight}
In this preliminary section we prove some auxiliary results on weight functions that will be used later in this work. We also discuss the special case when the weight function arises as the associated function of a weight sequence. Another class of weight functions will be considered in Section \ref{section subadditive}.

A weight function  is simply a non-decreasing function $\omega: [0,\infty) \rightarrow [0,\infty)$. Unless otherwise stated, we shall always assume throughout Sections \ref{sect-weight}-\ref{sect-support}  that $\omega$ satisfies

\begin{equation} \lim_{t \to \infty} \frac{\omega(t)}{\log t} = \infty.
\label{fasterthanlog}
\end{equation}
We often consider the ensuing additional conditions on weight functions:  
\begin{itemize}
\item[$(\delta)\phantom-$] $\displaystyle 2\omega(t) \leq \omega(Ht) + \log A$, $t \geq 0$, for some $A, H \geq 1$,
\item[$(\epsilon)_{0\phantom0}$] $\displaystyle \int_0^\infty \omega(t) e^{-\mu t}\dt < \infty$ for all $\mu > 0$, 
\item[$(\epsilon)_{\infty}$]  $\displaystyle \int_0^\infty \omega(t) e^{-\mu t}\dt < \infty$ for some $\mu > 0$. 
\end{itemize}
We also introduce the following quantified version of $(\epsilon)_0$ and $(\epsilon)_\infty$:
\begin{itemize}
\item[$(\epsilon)_{\mu\phantom0}$]  $\displaystyle \int_0^\infty \omega(t) e^{-\mu t}\dt < \infty$, \qquad  $\mu > 0$.
\end{itemize}
We extend $\omega$ to the whole real line by setting $\omega(t) = \omega(|t|)$, $t \in \R$. Furthermore, for $\lambda > 0$ we employ the short-hand notation $\omega_\lambda(t) = \omega(\lambda t)$. 

The relation $\omega \subset \sigma$ between two weight functions means that there are $C,\lambda>0$ such that 
$$
\sigma(t) \leq \omega_{\lambda}(t) + C, \qquad t \geq 0.
$$
The stronger relation $\omega \prec \sigma$ means that the latter inequality remains valid for every $\lambda>0$ and suitable $C=C_{\lambda}>0$. The reader should keep in mind that these two relations ``reverse'' orders of growth. We say that $\omega$ and $\sigma$ are \emph{equivalent}, denoted by $\omega \sim \sigma$, if both $\omega \subset \sigma$ and $\sigma \subset \omega$ hold.
\begin{examples} Some examples of weight functions are:
\begin{itemize}
\item $t^s, \quad s > 0,$ 
\item $ \exp(t^s\log^r(1+t)), \quad 0 \leq s < 1,\: r \geq 0, \: sr > 0,$
\item $\displaystyle\exp\left(\frac{t}{\log^s(e+t)}\right), \quad s > 0,$
\item $\ e^{t}$.
\end{itemize}
They all satisfy $(\delta)$. Moreover, the first three of them fulfill $(\epsilon)_0$,  while the last one satisfies $(\epsilon)_\infty$ but not $(\epsilon)_0$.  
\end{examples}

The next lemma gives alternative forms of $(\epsilon)_{0}$ and $(\epsilon)_{\infty}$.
\begin{lemma} \label{growth}
Let $\nu > \mu > 0$ and suppose $\omega$ is a weight function satisfying $(\epsilon)_\mu$, then $\omega(t) = o(e^{\nu t})$. Consequently, $\omega$ satisfies $(\epsilon)_0$ ($(\epsilon)_\infty$) if and only if $e^t \prec \omega$ ($e^t \subset \omega$).
\end{lemma}
\begin{proof}
Suppose the opposite, then there would exist $\varepsilon > 0$ and a sequence of positive numbers $(t_n)_{n \in \N}$  such that
\[ \omega(t_n) \geq \varepsilon e^{\nu t_n}, \qquad t_{n+1} \geq \frac{\nu t_n}{\mu}, \qquad  n \in \N.\]
Hence,
\[ \int_0^\infty \omega(t) e^{-\mu t}\dt \geq \sum_{n = 0}^\infty \int_{t_n}^{t_{n+1}} \omega(t) e^{-\mu t}\dt \geq \left(\frac{\nu}{\mu} - 1\right)t_0\sum_{n = 0}^\infty \omega(t_n) e^{-\nu t_n}.  \]
Since the last series is divergent, this contradicts $(\epsilon)_\mu$.
\end{proof}

We now show three useful lemmas.
\begin{lemma}\label{equivalent}
Let $\omega$ be a weight function satisfying $(\delta)$ and $(\epsilon)_0$ ($(\epsilon)_\infty$). Then, there is another weight function $\sigma$ with $\omega \sim \sigma$ that satisfies $(\delta)$, $(\epsilon)_0$ ($(\epsilon)_\infty$), and the additional condition
\begin{itemize}
\item[$(\zeta)$] $\displaystyle \lim_{t \to \infty} \sigma(\lambda t) - \sigma(t) = \infty, \qquad \forall \lambda > 1.$
\end{itemize}
\end{lemma}
\begin{proof}
Let $(t_n)_{n \in \N}$ be a sequence of non-negative numbers such that  $t_0 = 0 < t_1 < t_2 < \ldots <t_n\rightarrow \infty$ and 
$$
\omega(t) \geq n\log t, \qquad t \geq t_n.
$$
Define
$$
\rho(t) = n \log t, \qquad \mbox{if } t \in [t_n, t_{n+1}),
$$
and $\sigma(t) = \omega(t) + \rho(t) \leq 2\omega(t)$. The condition $(\delta)$ thus implies that $\omega$ and $\sigma$ are equivalent weight functions. Since $(\delta)$ and $(\epsilon)_0$ ($(\epsilon)_\infty$) are invariant under the relation $\sim$, the weight $\sigma$ satisfies these conditions as well. For $\lambda > 1$ and $t \in [t_n, t_{n+1})$, we have
$$
\sigma(\lambda t) - \sigma(t)\geq  \rho(\lambda t) - \rho(t) \geq n(\log (\lambda t) - \log t) = n \log \lambda, 
$$
whence $(\zeta)$ follows.
\end{proof}

\begin{lemma} \label{majorizationweight}
Let $\omega$ be a weight function satisfying $(\epsilon)_0$  ($(\epsilon)_\infty$). Then, there is a weight function $\sigma$ satisfying $(\epsilon)_0$ ($(\epsilon)_\infty$), $(\delta)$, and $\omega(t)\leq \sigma(t)$ for all $t\geq0$.
\end{lemma}
\begin{proof} We take $\sigma(t)=\int_{0}^{t+1} \omega(x)\mathrm{d}x$. The condition $(\epsilon)_0$ ($(\epsilon)_\infty$) clearly holds for $\sigma$. We also have $\sigma(t)\geq \int_{t}^{t+1}\omega(x)\mathrm{d}x\geq \omega(t)$. Furthermore, since $\omega(t)$ is non-decreasing, we obtain that $\sigma$ is convex; therefore, $2\sigma(t)\leq \sigma(2t)+\sigma(0)$.
\end{proof}

\begin{lemma} \label{weight}
Let $\omega$ be a weight function satisfying $(\epsilon)_0$. Then, there is a weight function $\sigma$ satisfying $(\epsilon)_0$ such that $\omega_\lambda(t) = o(\sigma(t))$ for all $\lambda > 0$.
\end{lemma}
\begin{proof}
We inductively determine a sequence of non-negative numbers $(t_n)_{n \in \Z_+}$ with $t_1 = 0$ that satisfies 
\[ \int_{t_{n}}^\infty \omega(t)e^{-t/n^2}\dt \leq \frac{1}{2^n}, \qquad \frac{t_n}{n} \geq \frac{t_{n-1}}{n-1} + 1, \qquad n \geq 2.  \]
We now define
\[ \sigma(t) = n\omega(nt), \qquad \mbox{for } t \in \left[\frac{t_n}{n}, \frac{t_{n+1}}{n+1}\right). \]
Clearly, $\sigma$ is a weight function and  $\omega_\lambda(t) = o(\sigma(t))$ for all $\lambda > 0$. Moreover, for each $n_0 \in \Z_+$,
\begin{align*}
 \int_0^\infty \sigma(t) e^{-t/n_0}\dt 
&\leq  \int_0^{t_{n_0}/n_0} \sigma(t) e^{-t/n_0}\dt + \sum_{n = n_0}^\infty \int_{t_{n}/n}^{t_{n+1}/n+1} n\omega(nt) e^{-t/n}\dt \\
&\leq t_{n_0}\omega(t_{n_0}) +   \sum_{n = n_0}^\infty \int_{t_n}^\infty \omega(t) e^{-t/n^2}\dt \\
&\leq t_{n_0}\omega(t_{n_0}) +   \sum_{n = n_0}^\infty  \frac{1}{2^{n}}  < \infty. 
\end{align*}
\end{proof}

We now consider the case when the weight function arises as the associated function of a weight sequence \cite{Komatsu}. Let $(M_p)_{p \in \N}$ be a sequence of positive real numbers and define
$m_p = M_p / M_{p-1}$ for all $p \geq 1$.  We shall always assume that $\lim_{p \rightarrow \infty} m_p = \infty$ and that the sequence $M_p$ is log-convex,  that is, $M_p^2 \leq M_{p-1}M_{p+1}$ for $p \geq 1$, or, equivalently, that $m_{p}$ is non-decreasing. In the sequel, we consider the ensuing conditions for weight sequences:
\begin{itemize}
\item[$(M.2)\phantom -$] $\displaystyle M_{p+q}\leq A H^{p+q} M_{p} M_{q}, p,q \in \mathbb{N}$, for some $A, H \geq 1$,
\item[$(M.5)_{0\phantom 0}$] $\displaystyle \sum_{n = 1}^\infty e^{-\mu m_p} < \infty$ for all $\mu > 0$,
\item[$(M.5)_\infty$] $\displaystyle \sum_{n = 1}^\infty e^{-\mu m_p} < \infty$ for some $\mu > 0$.
 \end{itemize}

The associated function $M$ of $M_p$ is defined as
$$ 
M(t) = \sup_{p \in \N} \log \left(\frac{t^pM_0}{M_p}\right), \qquad t > 0,
$$
and $M(0)=0$.
Clearly, $M$ is a weight function. 
We denote by
$$
m(t) = \sum_{m_p \leq t} 1, \qquad t \geq 0,
$$
the counting function of the sequence $(m_p)_{p\in\mathbb{Z}_+}$. It is well known that  \cite[Eq.\ (3.11), p.\ 50]{Komatsu} 
\begin{equation} 
M(t) = \int_0^t \frac{m(\lambda)}{\lambda} \dl, \qquad t \geq 0. 
\label{counting}
\end{equation}
As customary, the relation $M_p\subset N_p$ between two weight sequences means that there are $C,h>0$ such that 
$M_p\leq Ch^{p}N_{p}$ for all $p\in\mathbb{N}$. The stronger relation $M_p\prec N_p$ means that the latter inequality remains valid for every $h>0$ and a suitable $C=C_{h}>0$. We say that $M_p$ and $N_p$ are \emph{equivalent}, denoted by $M_p \sim N_p$, if both $M_p\subset N_p$ and $N_p\subset M_p$ hold. Denote by $M$ and $N$ the associated functions of $M_p$ and $N_p$, respectively. Then, $M_p\subset N_p$ ($M_p \prec N_p$) if and only if $M\subset N$ ($M \prec N$) \cite[Lemmas 3.8 and 3.10]{Komatsu} and therefore our use of the symbols $\subset$, $\prec$, and $\sim$ for weight functions is consistent with that for weight sequences.

\begin{examples}\label{ex1sequences}Some examples of weight sequences are:  
\begin{itemize}
\item $p!^{s}$, \quad $s > 0$, 
\item $\log(p+e)^{s(p+e)^r}$, \quad  $s,r \geq 1$, $sr > 1$,
\item $\log(p+e)^{(p+e)}$.
\end{itemize}
They all satisfy $(M.2)$. Moreover, the first two of them fulfill $(M.5)_0$  while the last one satisfies $(M.5)_\infty$ but not $(M.5)_0$.  
\end{examples}

The next proposition characterizes $(M.5)_{0}$ and $(M.5)_{\infty}$ in terms of the associated function.
\begin{proposition}\label{pro-epsilon-sequences}
Let $M_p$ be a weight sequence.  Then, $M$ satisfies $(\delta)$ if and only if $M_p$ satisfies $(M.2)$. Moreover, the following statements are equivalent: 
\begin{itemize}
\item[$(i)$] $M$ satisfies $(\epsilon)_0$ ($(\epsilon)_\infty$),
\item[$(ii)$] $m$ satisfies $(\epsilon)_0$ ($(\epsilon)_\infty$),
\item[$(iii)$] $e^t  \prec M$ ($e^t  \subset M$),
\item[$(iv)$] $e^t  \prec m$ ($e^t  \subset m$),
\item[$(v)$] $M_p$ satisfies $(M.5)_0$ ($(M.5)_\infty$).
\item[$(vi)$] $\log(p+e)^{(p+e)} \prec M_p$ ($\log(p+e)^{(p+e)} \subset M_p$).
\end{itemize}
\end{proposition}
\begin{proof}
The equivalence between $(M.2)$ and $(\delta)$ is shown in \cite[Prop.\ 3.6]{Komatsu}. Integration by parts yields
\[ \sum_{m_p \leq t } e^{-\mu m_p} = \int_0^t e^{-\mu \lambda}  \dml = m(t)e^{-\mu t} + \mu  \int_0^t m(\lambda) e^{-\mu \lambda} \dl, \qquad \mu > 0.\]
Moreover, by \eqref{counting}, we obtain
\[ \int_0^t M(\lambda)e^{-\mu \lambda}  \dl = \frac{1}{\mu}  \int_0^t \frac{m(\lambda) e^{-\mu \lambda}}{\lambda} \dl - \frac{M(t)e^{-\mu t}}{\mu}, \qquad \mu > 0.\]
Lemma \ref{growth} now implies that $(i)$--$(v)$ are equivalent to one another. Since the associated function of the sequence $\log(p+e)^{(p+e)}$ is equivalent to $e^t$, conditions $(iii)$ and $(vi)$ are also equivalent to each other.
\end{proof}
\section{Analytic functions in a strip}\label{harmonic}
The goal of this section is to construct functions that are analytic and satisfy certain lower and upper bounds with respect to a weight function in a given horizontal strip of the complex plane. Since the functions we aim to construct are zero-free, we first study harmonic functions in a strip. We also show a Phragm\'en-Lindel\"of type result for analytic functions defined on strips and having decay with respect to a weight function, Proposition \ref{praghmenlindelof} actually delivers a useful three lines type inequality.  As usual,  $\mathcal{O}(V)$ stands for the space of analytic functions in an open set $V\subseteq \mathbb{C}$.

Given $h > 0$, we write $T^h =\R + i(-h,h)$, $T_+^h =\R + i(0,h)$, and $T^h_-=\R + i(-h,0)$. Furthermore, we shall always write $z = x+iy \in \C$ for a complex variable. The Poisson kernel of the strip $T^\pi_+$ is well-known. It is given by
\[ P(x,y) = \frac{\sin y}{\cosh x - \cos y}, \]
and has the ensuing properties \cite{Widder}:
\begin{enumerate}
\item[$(A)$]   $P(x,y)$ is harmonic on $T^{2\pi}_+$,
\item[$(B)$]   $P(x,y) > 0$ on $T^{\pi}_+$,
\item[$(C)$]  $\displaystyle |P(x,y)| \leq \frac{|\sin y|e^{-|x| + 1}}{\cosh1 -1}, \qquad  |x| \geq 1,\:  0 < y < 2\pi,$
\item[$(D)$]  $\displaystyle \int_{0}^\infty P(x,y)\dx = \pi - y, \qquad 0 < y < \pi$.
\end{enumerate}

We employ the notation
$$P_h(x,y) = P\left( \frac{\pi x}{h}, \frac{\pi y}{h}\right), \qquad h > 0,$$
and, for a measurable function $f$ on $\mathbb{R}$,
$$
P_h\{f; x,y\} = \frac{1}{2h} \int_{-\infty}^\infty P_h(t-x,y) f(t)\dt,
$$ 
its Poisson transform with respect to the strip $T_{+}^{h}$.
\begin{lemma} \label{fund}
Let $\omega$ be a weight function satisfying $(\epsilon)_{\pi/h}$.Then, $P_h\{\omega; x,y\}$ is harmonic on $T_+^h$ and satisfies the lower bound
$$
P_h\{ \omega; x,y\} \geq \frac{\omega(x)}{2}\left(1 - \frac{y}{h}\right), \qquad  x+iy \in T_+^h.
$$
If $\omega$ satisfies $(\epsilon)_{\pi/(2h)}$, then the upper bound
\[ P_h\{\omega; x,y\} \leq  \left(\omega(2x) + \omega\left(\frac{2h}{\pi}\right)\right)\left(1 - \frac{y}{h}\right) + C , \qquad  x+iy \in T^{h}_+,\]
holds as well, where
$$
C = \frac{e}{2h(\cosh 1 - 1)} \int_{0}^\infty e^{-\pi t/(2h)}\omega(t) \dt < \infty.
$$
\end{lemma}
\begin{proof}
Since
$$
P_h\{\omega; x,y\} = P_\pi\left \{\omega_{h/\pi}; \frac{\pi x}{h}, \frac{\pi y}{h}\right \},
$$
we may assume that $h = \pi$. Set $P\{\omega; x,y\}  = P_\pi\{\omega; x,y\}$. The fact that $\omega$ satisfies $(\epsilon)_{1}$ implies that $P\{\omega ; x,y\}$ is harmonic on $T^{\pi}_+$ \cite[Thm.\ 1]{Widder}. By the symmetry properties of the weight function $\omega$ and the Poisson kernel $P$, it suffices to show the inequalities for $x \geq 0$. We have that
\begin{align*}
P\{\omega; x,y\} &= \frac{1}{2\pi}\int_{-\infty}^\infty P(t,y) \omega(t+x)\dt 
\geq \frac{1}{2\pi}\int_{0}^\infty P(t,y) \omega(t+x)\dt \\
& \geq \frac{\omega(x)}{2\pi}  \int_{0}^\infty P(t,y)\dt
= \frac{\omega(x)}{2}\left(1 - \frac{y}{\pi}\right) ,\\
\end{align*}
where, in the last equality, we have used property $(D)$ of the Poisson kernel $P$.
Next, assume that $\omega$ satisfies $(\epsilon)_{1/2}$. Then,
\begin{align*}
P\{\omega ; x,y\} 
&\leq \frac{1}{\pi}\int_{0}^\infty P(t,y) \omega(t+x)\dt \\
&=\frac{1}{\pi}\int_{0}^x P(t,y) \omega(t+x)\dt + \frac{1}{\pi}\int_{x}^\infty P(t,y) \omega(t+x)\dt \\
& \leq \frac{\omega(2x)}{\pi}\int_{0}^\infty P(t,y)\dt + \frac{1}{\pi}\int_{0}^\infty P(t,y) \omega(2t)\dt \\
& \leq \omega(2x)\left(1 - \frac{y}{\pi}\right)+ \frac{1}{\pi}\int_{0}^\infty P(t,y) \omega(2t)\dt. \\
\end{align*}
 Property $(C)$ of the Poisson kernel $P$ and condition $(\epsilon)_{1/2}$ imply that
\begin{align*}
\int_{0}^\infty P(t,y) \omega(2t)\dt &= \int_{0}^1 P(t,y) \omega(2t)\dt + \int_{1}^\infty P(t,y) \omega(2t)\dt \\
&\leq  \omega(2)(\pi -y) + \frac{e}{2(\cosh1 -1)} \int_{0}^\infty e^{-t/2} \omega(t)\dt < \infty.
\end{align*}
\end{proof}

Lemma \ref{fund} has the following important consequence.
\begin{proposition}\label{funda}
Let $\omega$ be a weight function satisfying $(\epsilon)_{\pi/(8h\lambda)}$. Then, there is $F \in \mathcal{O}(T^h)$ such that
$$
e^{\omega(\lambda x)} \leq |F(z)| \leq  Ce^{4\omega(2\lambda x)}, \qquad  z  \in T^h,
$$
for some $C=C_{h,\lambda} > 0$. If, in addition, $\omega$ satisfies $(\delta)$, then
$$
|F(z)| \leq  A^3Ce^{\omega(2H^2\lambda x)}, \qquad  z \in T^h,
$$
where $A,H$ are the constants occurring in $(\delta)$.
\end{proposition}
\begin{proof}
Define $U(x,y) = 4P_{4h}( \omega_\lambda ; x,y + h)$. Lemma \ref{fund} implies that $F(z) = e^{U(x,y) + iV(x,y)}$, with $V$ the harmonic conjugate of $U$, satisfies all requirements.
\end{proof}
\begin{remark} \label{remarksubadd}
Let $\omega$ be a weight function (not necessarily satisfying \eqref{fasterthanlog}) that is subadditive, i.e.
$$
\omega(t_1 +t_2) \leq \omega(t_1) + \omega(t_2), \qquad t_1, t_2 \geq 0.
$$
Observe that subadditivity implies that $\omega(t) = O(t)$; in particular, $(\epsilon)_0$ holds. One can readily show the inequalities (cf.\  the proof of Lemma \ref{fund})
$$
\frac{\omega(x)}{2}\left(1 - \frac{y}{h}\right) \leq P_h\{\omega; x,y\} \leq  \left(\omega(x) + \omega\left(\frac{h}{\pi}\right)\right)\left(1 - \frac{y}{h}\right) + C , \qquad  x+iy \in T^{h}_+,
$$
with
$$
C = \frac{e}{h(\cosh 1 - 1)} \int_{0}^\infty e^{-\pi t/h}\omega(t) \dt < \infty.
$$
Hence, for each $\lambda > 0$ and $h > 0$, there is $F \in \mathcal{O}(T^h)$ such that
$$
e^{\lambda\omega(x)} \leq |F(z)| \leq  Ce^{4\lambda\omega(x)}, \qquad  z \in T^h,
$$
for some $C > 0$. Subadditive weight functions will play an important role in Sections \ref{section subadditive} and \ref{boundary ultra exponential type} below.
\end{remark}
We end this section with a Phragm\'en-Lindel\"of type result for analytic functions defined on strips.
 We need the following lemma.
\begin{lemma}\label{subharmonicstrip}
Let $\varphi$ be analytic and bounded on the strip $T^h_+$ and continuous on $\overline{T^h_+}$. Then, if $\varphi$ is non-identically zero,
\begin{itemize}
\item[$(i)$] $ -\infty < \displaystyle \int_{-\infty}^\infty  \log|\varphi(x)|e^{-\pi |x|/h} \dx$ and $ -\infty < \displaystyle \int_{-\infty}^\infty  \log|\varphi(x + ih)|e^{-\pi |x|/h} \dx$,
\item[$(ii)$] $\log |\varphi(z)| \leq P_h\{\log |\varphi|; x,y\} + P_h\{\log |\varphi(\cdot + ih)|; x,h - y\}$, \quad $z \in T^h_+$.
\end{itemize}
\end{lemma}
\begin{proof} 
We may assume that $h = \pi$. The results can be derived by considering the conformal mapping 
$$z \rightarrow \frac{i-e^z}{i+e^z}$$
from the strip $T^\pi_+$ onto the the unit disk and using the well known counterparts of the statements for bounded analytic functions on the unit disk \cite{Koosis}. \end{proof}

\begin{proposition} \label{praghmenlindelof}
Let $\omega$ be a weight function satisfying $(\epsilon)_{\pi/h}$. Let $\varphi$ be holomorphic on the strip $T^h_+$ and continuous on $\overline{T^h_+}$,  and suppose that
$$
|\varphi(z)| \leq M, \quad z \in T^h_+, \qquad  |\varphi(x)| \leq Ce^{-\omega(x)}, \quad x \in \R,
$$
for some $M, C > 0$. Then,
$$
|\varphi(z)| \leq M^{y/h}C^{1 -(y/h)}\exp\left(-\frac{\omega(x)}{2}\left(1-\frac{y}{h}\right) \right), \qquad z = x+iy \in T^h_+.
$$
\end{proposition}
\begin{proof}
We may assume that $\varphi$ is non-identically zero.  By Lemma \ref{subharmonicstrip}$(ii)$ and property $(D)$ of the Poisson kernel we have that
$$
|\varphi(z)|e^{P_h\{\omega; x,y\}}\leq M^{y/h}\exp\left( \frac{1}{2h} \int_{-\infty}^\infty (\log|\varphi(t)| +\omega(t))P_h(t-x,y)\dt\right).
$$
By applying Jensen's inequality to the probability measure $P_h(t-x,y)/(2(h-y)) dt$ for $x \in \R$ and $0 < y < h$ fixed, we obtain
\begin{align*}
|\varphi(z)|e^{P_h\{\omega; x,y\}}&\leq \frac{M^{y/h}}{2(h -y)}\int_{-\infty}^\infty \exp\left(\left(1 -\frac{y}{h}\right)(\log|\varphi(t)| +\omega(t))\right)P_h(t-x,y)\dt \\
& \leq M^{y/h}C^{1 -(y/h)}.
\end{align*}
The result now follows from the first part of Lemma \ref{fund}.
\end{proof}
\section{Spaces of analytic functions of rapid decay on strips}\label{testfunctions}
We now introduce and discuss some basic properties of the spaces of analytic functions that we are mainly concerned with. They generalize the test function spaces for the Fourier hyperfunctions and ultrahyperfunctions \cite{Hoskins,Kaneko,Park,Silva}. We will also determine their images under the Fourier transform, extending various results from \cite{c-k-l-1997}. Throughout the rest of the article the parameters $h$, $k$,  $\lambda$, $b$, and $R$ always stand for positive real numbers.

Let $\omega$ be a weight function. We write $\mathcal{A}^h_\omega$ for the $(B)$-space (Banach space) consisting of all analytic functions $\varphi \in \mathcal{O}(T^h)$ that satisfy
$$
 \| \varphi \|^h_{\omega} := \sup_{z \in T^h}|\varphi(z)|e^{\omega(x)} < \infty.
$$
We set $\mathcal{A}^h_{\omega_\lambda} = \zw^{h,\lambda}$ and
$$\| \varphi \|^h_{\omega_\lambda} = \| \varphi \|^{h, \lambda}_{\omega} = \| \varphi \|^{h,\lambda}, \qquad \varphi \in \zw^{h,\lambda}.$$

Montel's theorem and the uniqueness property of holomorphic functions yield the following simple lemma.
 \begin{lemma}\label{compact}
Let $\omega$ and $\sigma$ be two weight functions such that 
\begin{equation}
\label{eqomega-sigma}
\lim_{t \to \infty}\sigma(t) - \omega(t) = \infty.
\end{equation}
Then, for $0 < h < k$, the restriction mapping $\mathcal{A}^k_\sigma \rightarrow \zw^{h}$ is injective and compact.
\end{lemma}

As topological vector spaces, we define 
\begin{equation}
\label{eq1mainspaces}
\mathcal{U}_{(\omega)}(\C) = \varprojlim_{ h \to \infty}\varprojlim_{ \lambda \to \infty} \mathcal{A}_\omega^{h,\lambda}, \qquad  \mathcal{A}_{\{\omega\}}(\R) =  \varinjlim_{ h \to 0^+} \varinjlim_{ \lambda \to 0^+}\mathcal{A}_\omega^{h,\lambda},
\end{equation}
and 
\begin{equation}
\label{eq2mainspaces}
\mathcal{A}^h_{(\omega)} = \varprojlim_{ k \to h^-}\varprojlim_{ \lambda \to \infty} \mathcal{A}_\omega^{k,\lambda}, \qquad \mathcal{A}^h_{\{\omega\}} = \varinjlim_{ k \to h^+} \varinjlim_{ \lambda \to 0^+}\mathcal{A}_\omega^{k,\lambda}.
\end{equation}
If $\omega$ satisfies $(\delta)$, Lemma \ref{compact} implies that $\mathcal{U}_{(\omega)}(\C)$ and $\mathcal{A}^h_{(\omega)}$ are $(FS)$-spaces, while $\mathcal{A}_{\{\omega\}}(\R)$ and $\mathcal{A}^h_{\{\omega\}}$ are $(DFS)$-spaces. Let $\sigma$ be another weight function such that $\omega \sim \sigma$, then $\mathcal{U}_{(\omega)}(\C) = \mathcal{U}_{(\sigma)}(\C)$ and $\mathcal{A}_{\{\omega\}}(\R)= \mathcal{A}_{\{\sigma\}}(\R)$, topologically. The same is true for the spaces (\ref{eq2mainspaces}).
We shall call the elements of the dual spaces $\mathcal{U}'_{(\omega)}(\C)$ and $\mathcal{A}'_{\{\omega\}}(\R)$ \emph{ultrahyperfunctions of $(\omega)$-type} and \emph{hyperfunctions of $\{\omega\}$-type}, respectively. When no reference to $\omega$ is made, we will simply call them (ultra)hyperfunctions of fast growth. We endow these (and other) duals with the strong topology.
 Observe that $\mathcal{U}_{(t)}(\C)$ and $\mathcal{A}_{\{t\}}(\R)$ are the test function spaces for the Fourier ultrahyperfunctions \cite{Hoskins, Park} and Fourier hyperfunctions \cite{Kawai}, respectively. If $\omega = M$ is the associated function of a weight sequence $M_p$, then $ \mathcal{U}_{(M_p)}(\C) := \mathcal{U}_{(M)}(\C) = \mathcal{S}^{(p!)}_{(M_p)}(\R)$ and  $\mathcal{A}_{\{M_p\}}(\R) := \mathcal{A}_{\{M\}}(\R) = \mathcal{S}^{\{p!\}}_{\{M_p\}}(\R)$, the mixed type Gelfand-Shilov spaces \cite{PPV}. In particular, $\mathcal{A}_{\{p!^\alpha\}}(\R)$, $\alpha > 0$, is equal to the Gelfand-Shilov space $\mathcal{S}^{\{1\}}_{\{\alpha\}}$ \cite{Gelfand}. 

As already pointed out in the introduction, it is a priori not obvious that the spaces (\ref{eq1mainspaces}) and (\ref{eq2mainspaces}) should contain non-identically zero functions. We address in this section the non-triviality of the spaces (\ref{eq2mainspaces}), and in particular that of $\mathcal{A}_{\{\omega\}}(\mathbb{R})=\bigcup _{h>0}\mathcal{A}_{\{\omega\}}^{h}$. The analysis of the corresponding problem for $\mathcal{U}_{(\omega)}(\mathbb{C})$ is more demanding and is postponed to Section \ref{section boundary values}. We begin with the following necessary condition for the non-triviality of $\mathcal{A}_\omega^{h,\lambda}$.
\begin{proposition} \label{mandel}
Let $\omega$ be a weight function and suppose that $\mathcal{A}_\omega^{h,\lambda}$ contains a function which is non-identically zero. Then, $\omega$ satisfies $(\epsilon)_{\pi/(h\lambda)}$.
\end{proposition}
\begin{proof}
Let $\varphi$ be a non-zero element of $\mathcal{A}_\omega^{h,\lambda}$. By applying Lemma~\ref{subharmonicstrip}$(i)$ to $\varphi(\cdot - ik)$, $0 < k < h$, we obtain
\begin{align*}
- \infty &< \int_{-\infty}^ \infty \log |\varphi(x - ik)|e^{-\pi |x|/h}\dx 
\\
&
\leq \log \| \varphi \|^{h,\lambda} \int_{-\infty}^ \infty e^{-\pi |x|/h}\dx - \frac{2}{\lambda} \int_0^\infty \omega(x)e^{-\pi x/(h\lambda)}\dx, 
\end{align*}
and so $\int_0^\infty \omega(x)e^{-\pi x/(h\lambda)}\dx<\infty$.
\end{proof}
\begin{proposition}\label{localnontriviality}
Let $\omega$ be a weight function. The space $\mathcal{A}^h_{(\omega)}$ ($\mathcal{A}^h_{\{\omega\}}$) is non-trivial if and only if $\omega$ satisfies $(\epsilon)_{0}$ ($(\epsilon)_{\infty}$).
Consequently, $\mathcal{A}_{\{\omega\}}(\R)$ is non-trivial if and only if $\omega$ satisfies $(\epsilon)_{\infty}$.
\end{proposition}
\begin{proof}
The direct implication follows from Proposition \ref{mandel}. Moreover, if $\omega$ satisfies $(\epsilon)_\infty$, Proposition \ref{funda} gives the non-trivality of $\mathcal{A}^h_{\{\omega\}}$. Assume now that $\omega$ satisfies $(\epsilon)_0$. By Lemma \ref{weight}, there is a weight function $\sigma$ satisfying $(\epsilon)_0$ such that $\omega_\lambda(t) = o(\sigma(t))$ for all $\lambda > 0$. By Proposition \ref{funda} there is an analytic function $F$ on $T^h$ such that $|F(z)| \geq e^{\sigma(x)}$ for all $z \in T^h$. Then, $1/F$ is an element of $\mathcal{A}^h_{(\omega)}$ which is non-identically zero. 
\end{proof}
In the next section (Subsection \ref{sect-cohBeurling}) we shall show that $\mathcal{U}_{(\omega)}(\C)$ is non-trivial if and only if $\omega$ satisfies $(\epsilon)_{0}$, this will be a consequence of a representation theorem for its dual (see Theorem \ref{nontrivial2}). 

The rest of this section is devoted to computing the images of  $\mathcal{U}_{(\omega)}(\C)$ and $\mathcal{A}_{\{\omega\}}(\R)$ under the Fourier transform. These spaces are the ultradifferentiable counterparts of the classical Hasumi-Silva space $\mathcal{K}_1(\R)$ of exponentially rapidly decreasing smooth functions \cite{Hoskins,hasumi,zielezny}. We fix constants in the Fourier transform as 
$$
\mathcal{F}(\varphi)(\xi) = \widehat{\varphi}(\xi) = \int_{-\infty}^\infty \varphi(x) e^{-ix \xi } \dx.
$$
Let $\omega$ be a weight function. We write  $\mathcal{K}^h_{1,\omega}(\R)$ for the function space consisting of all $\psi \in L^1(\R)$ such that
$$
\rho_\omega(\psi) := \sup_{x \in \R} |\mathcal{F}^{-1}(\psi)(x)|e^{\omega(x)} < \infty, \qquad \rho^h(\psi) := \sup_{\xi \in \R} |\psi(\xi)|e^{h|\xi|} < \infty;
$$
it becomes a $(B)$-space when endowed with the norm
$$
\rho^h_\omega(\psi) :=  \max (\rho_\omega(\psi),\rho^h(\psi)), \qquad \psi \in \mathcal{K}^h_{1,\omega}(\R).
$$
We set further $\mathcal{K}^h_{1,\omega_\lambda}(\R) = \mathcal{K}^{h,\lambda}_{1,\omega}(\R)$,
$$
\rho_{\omega_\lambda}^h(\psi) = \rho_{\omega}^{h,\lambda}(\psi)  , \qquad \psi \in \mathcal{K}_{1,\omega}^{h,\lambda}(\R),
$$
and
$$
\mathcal{K}_{1,(\omega)}(\R)= \varprojlim_{\lambda \to \infty} \varprojlim_{h \to \infty}\mathcal{K}^{h,\lambda}_{1,\omega}(\R), \qquad \mathcal{K}_{1,\{\omega\}}(\R)= \varinjlim_{\lambda \to 0^+} \varinjlim_{h \to 0^+}\mathcal{K}^{h,\lambda}_{1,\omega}(\R).  
$$
\begin{proposition}\label{fourier}
Let $\omega$ be a weight function satisfying $(\delta)$ and $(\varepsilon_0)$ ($(\epsilon)_\infty)$. The Fourier transform is a topological isomorphism of $\mathcal{U}_{(\omega)}(\C)$ onto $\mathcal{K}_{1,(\omega)}(\R)$ (of $\mathcal{A}_{\{\omega\}}(\R)$ onto $\mathcal{K}_{1,\{\omega\}}(\R)$).
\end{proposition} 
\begin{proof}
For  $\varphi \in \mathcal{A}^{h,\lambda}_{\omega}$ we have the following formulas for its Fourier transform  (see e.g. \cite[p.\ 167]{Hoskins})
$$
\widehat{\varphi}(\xi) =
\left\{
	\begin{array}{ll}
		\mbox{$\displaystyle \int_{-\infty}^\infty \varphi(x+ik)e^{-i(x+ik)\xi} \dx$},  &  \xi \leq 0, \\ \\
		\mbox{$\displaystyle \int_{-\infty}^\infty \varphi(x-ik)e^{-i(x-ik)\xi} \dx$},  & \xi \geq 0,
	\end{array}
\right.
$$
where $ 0  < k < h$. This shows that the Fourier transform is a well defined continuous mapping in both the Beurling and Roumieu case. Conversely, let $\psi \in \mathcal{K}^{h,\lambda}_{1,\omega}(\R)$. Then, there is $\varphi \in \mathcal{O}(T^h)$ with $\widehat{\varphi} = \psi$ such that 
$$
|\varphi(z)| \leq \frac{\rho^h(\psi)}{\pi(h-k)}, \qquad z \in T^{k},
$$   
where $ 0  < k < h$. Invoking Proposition \ref{praghmenlindelof} and condition $(\delta)$, we conclude that the inverse Fourier transform is also a  well defined continuous mapping in both the Beurling and Roumieu case. \end{proof}

We call the elements of $\mathcal{K}'_{1,(\omega)}(\R)$ and $\mathcal{K}'_{1,\{\omega\}}(\R)$ \emph{ultradistributions of class $(\omega)$ (of Beurling type) of exponential type} and \emph{ultradistributions of class $\{\omega\}$ (of Roumieu type) of infra-exponential type}, respectively. Observe that Proposition \ref{fourier} allows one to define the Fourier transform from $\mathcal{K}'_{1,(\omega)}(\R)$ onto $\mathcal{U}'_{(\omega)}(\C)$ and from $\mathcal{K}'_{1,\{\omega\}}(\R)$ onto $\mathcal{A}'_{\{\omega\}}(\R)$ via duality.

Let us remark that when $\omega = M$ is the associated function of a weight sequence $M_p$ satisfying $(M.2)$, then $ \mathcal{K}_{1,(M_p)}(\R) := \mathcal{K}_{1,(M)}(\R) = \mathcal{S}^{(M_p)}_{(p!)}(\R)$ and $\mathcal{K}_{1,\{M_p\}}(\R) := \mathcal{K}_{1,\{M\}}(\R) = \mathcal{S}^{\{M_p\}}_{\{p!\}}(\R)$, topologically. When $M_p$ is non-quasianalytic \cite{Komatsu}, we have the continuous and dense inclusions $\mathcal{D}^{(M_p)}(\R) \hookrightarrow \mathcal{K}_{1,(M_p)}(\R)$ and $\mathcal{D}^{\{M_p\}}(\R) \hookrightarrow \mathcal{K}_{1,\{M_p\}}(\R)$ and therefore the inclusions $\mathcal{K}'_{1,(M_p)}(\R) \subset {\mathcal{D}^{(M_p)}}'(\R)$ and $\mathcal{K}'_{1,\{M_p\}}(\R) \subset {\mathcal{D}^{\{M_p\}}}'(\R)$.

\section{Boundary values of analytic functions in spaces of (ultra)hyperfunctions of fast growth}
\label{section boundary values}
In this section we build an analytic representation theory for the spaces $\mathcal{U}'_{(\omega)}(\C)$ and $\mathcal{A}'_{\{\omega\}}(\R)$. We show that every ultrahyperfunction of $(\omega)$-type (hyperfunction of $\{\omega\}$-type) can be represented as the boundary value of an analytic function defined outside some strip (outside the real line) and satisfying certain growth bounds with respect to the weight function $\omega$. Silva obtained analytic representations of  ultrahyperfunctions via a careful analysis of the Cauchy-Stieltjes transform \cite{Hoskins, Silva}. We shall follow a similar approach, the functions constructed in Section \ref{harmonic} are essential for our method. Furthermore, we present an (ultra)hyperfunctional version of Painlev\'e's theorem on analytic continuation. These results allow us to express the dual spaces $\mathcal{U}'_{(\omega)}(\C)$ and $\mathcal{A}'_{\{\omega\}}(\R)$ as quotients of spaces of analytic functions in a very precise fashion.  As an application of these ideas, we characterize the non-triviality of the space $\mathcal{U}_{(\omega)}(\C)$.

\subsection{Analytic representations}
Given $0 < b < R$, we use the notation $T^{b,R} = T^R \backslash \overline{T^b}= \mathbb{R}+i ((-R,-b)\cup (b,R))$.  Let $\omega$ be a weight function. We define $\mathcal{O}_{\omega}^{b,R}$ as the $(B)$-space consisting of all analytic functions $F \in \mathcal{O}(T^{b,R})$ that satisfy
$$
|F|^{b,R}_\omega := \sup_{z \in T^{b,R }} |F(z)| e^{-\omega(x)} < \infty,
$$
and $\pw^{R}$ as the $(B)$-space consisting of all $P \in \mathcal{O}(T^{R})$ such that
$$
|P|^{R} := \sup_{z \in T^R} |P(z)| e^{-\omega(x)} < \infty.
$$
We set $\mathcal{O}^{b,R}_{\omega_\lambda} = \mathcal{O}^{b,R, \lambda}_{\omega}$, $\mathcal{P}^{R}_{\omega_\lambda} = \mathcal{P}^{R, \lambda}_{\omega}$,
$$
| F |^{b,R}_{\omega_\lambda} = | F |^{b,R,\lambda}_\omega= | F |^{b,R,\lambda}, \qquad F \in \mathcal{O}^{b,R, \lambda}_{\omega},
$$
and
$$
| P |^R_{\omega_\lambda} = | P |^{R,\lambda}_\omega= | P |^{R,\lambda}, \qquad P \in \mathcal{P}^{R, \lambda}_{\omega}.
$$
As in Lemma \ref{compact}, one easily obtains:
 \begin{lemma}\label{compactanalrepr}
Let $\omega$ and $\sigma$ be weight functions such that \eqref{eqomega-sigma} holds. 
Then, for $0 < b < c < L <  R$, the restriction mappings $\mathcal{O}^{b,R}_\omega \rightarrow \mathcal{O}^{c,L}_\sigma$ and $\mathcal{P}^{R}_\omega \rightarrow \mathcal{O}^{L}_\sigma$ are injective and compact.
\end{lemma}
 
Suppose that $\sigma$ is another weight function such that both \eqref{eqomega-sigma} and
$$
\int_0^\infty e^{\omega(t) - \sigma(t)}\dt < \infty
$$
hold. If $\omega$ satisfies $(\delta)$, the above conditions are fulfilled for $\omega = \omega_\lambda$ and $\sigma = \omega_{H\lambda}$ for each $\lambda > 0$.
Let $0 < b < h < R$. Given an analytic function $F \in \ow^{b,R}$, we associate to $F$ an element of $(\mathcal{A}^h_\sigma)'$ via the \emph{boundary value mapping}
$$
\langle \operatorname{bv}(F), \varphi \rangle = -\int_{\Gamma_{k}}F(z) \varphi(z)dz, \qquad  \varphi \in \mathcal{A}^h_\sigma, 
$$ 
where $b < k < h $  and $\Gamma_k$ is the (counterclockwise oriented) boundary of $T^k$. By Cauchy's integral theorem, the definition of $\operatorname{bv}(F)$ is independent of the chosen $k$ and  $f = \operatorname{bv}(F)$ is indeed a continuous linear functional on $\mathcal{A}^h_\sigma$. We say that $F$ is an \emph{analytic representation of $f$}. We have the following general result on the existence of analytic representations.
\begin{proposition}\label{analyticreprgeneral}
Let $0 < k < b < h < R$ and let $\omega$, $\sigma$, and $\kappa$ be three weight functions satisfying
\begin{equation}
\lim_{t \to \infty}\sigma(t) - \omega(t) = \infty, \qquad \lim_{t \to \infty}\kappa(t) - \sigma(t) = \infty,
\label{condition0}
\end{equation}
and
\begin{equation}
\int_0^\infty e^{\omega(t) - \sigma(t)}\dt < \infty, \qquad  \int_0^\infty e^{\sigma(t) - \kappa(t)}\dt < \infty.
\label{condition}
\end{equation}
Furthermore, suppose there is $P \in \mathcal{O}(T^R)$ such that
$$
C_1e^{\omega(x)} \leq |P(z)| \leq C_2e^{\sigma(x)}, \qquad  z \in T^R,
$$
for some $C_1, C_2 >0$. Then, every $f \in (\mathcal{A}^k_\omega)'$ is the boundary value of some element of $\mathcal{O}_\sigma^{b,R}$ on $\mathcal{A}_{\kappa}^h$, that is, there is $F \in \mathcal{O}_\sigma^{b,R}$ such that $\operatorname{bv}(F) = f$ on $\mathcal{A}_{\kappa}^h$.
\end{proposition}
\begin{proof}
Cauchy's integral formula yields 
\[ \varphi(\zeta) = \frac{1}{2\pi i P(\zeta)} \int_{\Gamma_{b}} \frac{\varphi(z)P(z)}{z -\zeta} \dz, \qquad  \zeta \in T^{k},\]
for each $\varphi \in \mathcal{A}_{\kappa}^h$.
Let $R_n(\zeta)$ be a sequence of Riemann sums converging to the integral in the right-hand side of the above expression. Then,
$R_n(\zeta)/(2\pi iP(\zeta)) \rightarrow \varphi(\zeta)$, as $n \to \infty$ in  $\zw^{k}$, as one readily verifies. Hence,
$$ 
\langle f(\zeta), \varphi(\zeta) \rangle = \int_{\Gamma_{b}}\frac{P(z)}{2\pi i} \left \langle f(\zeta), \frac{1}{(z -\zeta)P(\zeta)} \right \rangle \varphi(z) \dz.  
$$
Thus,
\[ F(z) = \frac{P(z)}{2\pi i} \left \langle f(\zeta), \frac{1}{(\zeta -z)P(\zeta)}\right \rangle \]
is an element of $\mathcal{O}_\sigma^{b,R}$ such that $\operatorname{bv}(F) = f$  on $\mathcal{A}_{\kappa}^h$.
\end{proof}
Our next result shows that functions whose boundary value give rise to the zero functional can be analytically continued. 
\begin{proposition}\label{edgegeneral}
Let $0 < b < h < R$ and let $\omega$, $\sigma$, and $\kappa$ be three weight functions satisfying \eqref{condition0} and \eqref{condition}. Furthermore, suppose there is $P \in \mathcal{O}(T^R)$ such that
$$
C_1e^{\sigma(x)} \leq |P(z)| \leq C_2e^{\kappa (x)}, \qquad  z \in T^R,
$$
for some $C_1, C_2 > 0$. If $F \in \ow^{b,R}$ is such that $\operatorname{bv}(F) = 0$ on $\mathcal{A}^h_\sigma$, then $F \in \mathcal{P}^{R}_\kappa$.
\end{proposition}
\begin{proof}
Let $0 < b < k < h < L < R$. It suffices to show that
\begin{equation}
\label{eqeoftw}
 F(z) = \frac{P(z)}{2\pi i} \int_{\Gamma_{L}} \frac{F(\zeta)}{(\zeta -z)P(\zeta)} \dzeta, \qquad z \in T^{h,L}. 
\end{equation}
Fix $z \in \C$ with $h <  \operatorname{Im} z  < L$; the case $-L <  \operatorname{Im} z  < -h$ is analogous. We denote by $\Gamma^+$ ($\Gamma^-$) the part of a contour $\Gamma$ in the upper (lower) half-plane. Cauchy's integral formula yields 
$$
F(z) = \frac{P(z)}{2\pi i} \left(\int_{\Gamma^+_{L}} \frac{F(\zeta)}{(\zeta -z)P(\zeta)} \dzeta - \int_{\Gamma^+_{k}} \frac{F(\zeta)}{(\zeta -z)P(\zeta)} \dzeta\right).
$$
Since $\zeta \rightarrow 1/((\zeta -z)P(\zeta)) \in \mathcal{A}^h_\sigma$, the assumption $\operatorname{bv}(F) = 0$ on $\mathcal{A}^h_\sigma$ implies that 
$$
\int_{\Gamma^+_{k}} \frac{F(\zeta)}{(\zeta -z)P(\zeta)} \dzeta =  - \int_{\Gamma^-_{k}} \frac{F(\zeta)}{(\zeta -z)P(\zeta)} \dzeta.
$$
 Furthermore, because $\zeta \rightarrow 1/((\zeta -z)P(\zeta))$ is analytic on the horizontal strip $-R <  \operatorname{Im} \zeta  < -b$, we have by Cauchy's integral theorem that
 $$
 \int_{\Gamma^-_{k}} \frac{F(\zeta)}{(\zeta -z)P(\zeta)} \dzeta =  \int_{\Gamma^-_{L}} \frac{F(\zeta)}{(\zeta -z)P(\zeta)} \dzeta.
$$
This shows (\ref{eqeoftw}).
\end{proof}
Combining these two results with Proposition \ref{funda}, we obtain the following corollaries.
\begin{corollary}\label{analquan}
Let $0 < k < b < h < R$ and let $\omega$ be a weight function satisfying $(\delta)$ and $(\epsilon)_0$.   For every $f \in (\mathcal{A}^{k,\lambda}_\omega)'$ there is  $F \in \mathcal{O}_\omega^{b,R,2H^2\lambda}$ such that $\operatorname{bv}(F) = f$ on $\mathcal{A}^{h,2H^3\lambda}_\omega$.
\end{corollary}
\begin{corollary}\label{edge}
Let $0 < b < h < R$ and let $\omega$ be a weight function satisfying $(\epsilon)_0$ and $(\delta)$. If $F \in \ow^{b,R,\lambda}$ is such that $\operatorname*{bv}(F) = 0$ on $\mathcal{A}^{h, H\lambda}_\omega$, then $F \in \pw^{R,2H^3\lambda}$.
\end{corollary}
\subsection{Analytic representations of ultrahyperfunctions of $(\omega)$-type  and the non-triviality of $\mathcal{U}_{(\omega)}(\mathbb{C})$} \label{sect-cohBeurling}
We start by studying the analytic representations of the dual of $\mathcal{A}^h_{(\omega)}$. Let us introduce some further notation. For $0 < h < R$ we write
$$
\mathcal{O}_{(\omega)}^{h,R} = \varinjlim_{ b \to h^-} \varinjlim_{ L \to R^+} \varinjlim_{\lambda \to \infty} \mathcal{O}_{\omega}^{b,L, \lambda}, \qquad \mathcal{P}_{(\omega)}^{R} = \varinjlim_{ L \to R^+} \varinjlim_{\lambda \to \infty} \mathcal{P}_{\omega}^{L, \lambda}.
$$
Lemma \ref{compactanalrepr} implies that, if $\omega$ satisfies $(\delta)$, the spaces $\mathcal{O}_{(\omega)}^{h,R}$  and $\mathcal{P}_{(\omega)}^{R}$ are $(DFS)$-spaces. 

Furthermore,  the boundary value mapping
$$
\operatorname{bv}: \mathcal{O}_{(\omega)}^{h,R} \rightarrow (\mathcal{A}^h_{(\omega)})'
$$
is well defined and continuous. We need the following lemma.
\begin{lemma}  \label{denseinf}
Let $\omega$ be a weight function satisfying $(\epsilon)_0$ and $(\delta)$. For each $\lambda > 0$ the space 
$$
\mathcal{A}_{(\omega)}^{h,\infty} = \varprojlim_{\mu \to \infty} \mathcal{A}^{h,\mu}_\omega
$$
is dense in the space $\mathcal{A}_\omega^{h,H\lambda}$ with respect to the norm $\| \cdot \|^{h,\lambda}$. 
\end{lemma}
\begin{proof}
Let $\varphi \in  \zw^{h,H\lambda}$. By Proposition \ref{localnontriviality}, $\mathcal{A}_{(\omega)}^{h+1}$ is non-trival, so select $\psi \in \mathcal{A}_{(\omega)}^{h+1}$ with $\psi(0) = 1$. Define $\varphi_n(z) = \varphi(z)\psi(z/n) \in \mathcal{A}_{(\omega)}^{h, \infty} $ for $n \geq 1$. Then,
$$
 \| \varphi - \varphi_n\|^{h, \lambda} = \sup_{z \in T^h} |\varphi(z)(1- \psi(z/n))|e^{\omega(\lambda x)} \leq A\|\varphi\|^{h,H\lambda} \sup_{z \in T^h} |1- \psi(z/n)|e^{-\omega(\lambda x)},
$$
which proves the result since $\psi(z/n) \to 1$, as $ n \to \infty$, uniformly on compact subsets of $T^{h+1}$.
\end{proof}

We can now show that $(\mathcal{A}_{(\omega)}^{h})'$ is isomorphic to the quotient space $ \mathcal{O}^{h,R}_{(\omega)} / \mathcal{P}^R_{(\omega)}$.
\begin{proposition}\label{Beurlinglocalcoh}
Let  $0 < h < R$ and let $\omega$ be a weight function satisfying $(\epsilon)_0$ and $(\delta)$. The following sequence
$$
0 \longrightarrow \mathcal{P}^R_{(\omega)} \longrightarrow \mathcal{O}^{h,R}_{(\omega)} \xrightarrow{\phantom,\operatorname{bv}\phantom,}  (\mathcal{A}_{(\omega)}^{h})' \longrightarrow 0
$$
is topologically exact. Moreover, for every $f \in (\mathcal{A}_{(\omega)}^{h})'$ one can find $0 < b < h$ and $\lambda > 0$ such that for every $R > h$ there is $F \in \mathcal{O}_{\omega}^{b,R, \lambda}$ such that $\operatorname{bv}(F)= f$. In addition, for every bounded set $B \subset (\mathcal{A}^h_{(\omega)})'$ there is a bounded set $A \subset \mathcal{O}^{h,R}_{(\omega)}$ such that $\operatorname{bv}(A) = B$.
\end{proposition}
\begin{proof}
In view of the Pt\'ak open mapping theorem \cite[Chap.\ 3, Prop.\ 17.2]{Horvath}, it suffices to show that the sequence is algebraically  exact. It is clear that $\mathcal{P}^R_{(\omega)} \subseteq \ker{\operatorname{bv}}$. Conversely, let $F \in \mathcal{O}^{h,R}_{(\omega)}$ and suppose $\operatorname{bv}(F) = 0$ on $\mathcal{A}_{(\omega)}^{h}$. Let $0 < b < h$, $L > R$, and $\lambda > 0$ be such that $F \in \mathcal{O}_{\omega}^{b,L,\lambda}$. Since $\mathcal{A}_{(\omega)}^{h,\infty}\subset \mathcal{A}_{(\omega)}^{h}$, Lebesgue's dominated convergence theorem and Lemma \ref{denseinf} imply that actually $\operatorname{bv}(F) = 0$ on $\mathcal{A}_{\omega}^{h,H^2\lambda}$. Hence, by Corollary \ref{edge}, we have that $F \in \mathcal{P}^{L,2H^4\lambda}_{\omega} \subset \mathcal{P}^R_{(\omega)}$. The second statement (and therefore also the surjectivity of the boundary value mapping) is a consequence of the Hahn-Banach theorem and Corollary \ref{analquan}. The last part follows from the general fact that for a surjective continuous linear mapping $T: E \rightarrow F$ between reflexive $(DF)$-spaces it holds that for every bounded set $B \subset F$ there is a bounded set $A \subset E$ such that $T(A) = B$; this follows from the fact that $T^t$ is an injective strict morphism and the bipolar theorem.
\end{proof}
We now proceed to show that $\mathcal{U}_{(\omega)}(\C)$ is non-trivial if and only if $\omega$ satisfies $(\epsilon)_{0}$. For it, we need some basic facts about projective spectra. Let $(X_n)_{n \in \N}$ be a sequence of topological spaces and let $u_n^{n+1}: X_{n+1} \to X_n$ be a continuous mapping for each $n \in \N$. Consider the projective limit $X$ of the spectrum $(X_n)_{n \in \N}$, that is,
$$
X = \{(x_n) \in \prod_{n \in \N} X_n \, : \, x_n = u_n^{n+1}(x_{n+1}), \quad n \in \N\},
$$
with the natural projection mappings $
u_j: X \rightarrow X_j:  (x_n) \to x_j.$ The spectrum is called \emph{reduced} if $u_j$ has dense range for each $j \in \N$. We shall employ the following well known result, due to De Wilde (see also \cite{Dubinsky} for the case of Banach spaces).
\begin{lemma}[\cite{DeWilde}]\label{intersection}
Let $(X_n)_{n \in \N}$ be a spectrum of complete metrizable topological spaces. If $u_n^{n+1}$ has dense range for each $n \in \N$, then the spectrum $(X_n)_{n \in \N}$ is reduced.  
\end{lemma}
\begin{theorem}\label{nontrivial2}
The space $\mathcal{U}_{(\omega)}(\C)$ is non-trivial if and only if $\omega$ satisfies $(\epsilon)_{0}$. If, in addition, $\omega$ satisfies $(\delta)$, then $\mathcal{U}_{(\omega)}(\C)$ is dense in $\mathcal{A}_{(\omega)}^{h}$ for all $h > 0$.
\end{theorem}
\begin{proof} 
The direct implication follows from Proposition \ref{mandel}. By Lemma~\ref{majorizationweight}, we may assume that $\omega$ satisfies $(\delta)$  for the first assertion. So, for the converse, first notice that Proposition \ref{localnontriviality} ensures that the spaces $\mathcal{A}_{(\omega)}^h$ are non-trivial for each $h>0$. Lemma \ref{intersection} implies that it suffices to show that for all $0 < k < h$ the space $\mathcal{A}_{(\omega)}^h$ is dense in $\mathcal{A}_{(\omega)}^{k}$, but this precisely follows from Proposition \ref{Beurlinglocalcoh} and the Hahn-Banach theorem.  
 \end{proof}
 
 Combining Lemma \ref{denseinf} and Theorem \ref{nontrivial2}, we obtain,
 
 \begin{corollary} The continuous inclusion $\mathcal{U}_{(\omega)}(\C) \hookrightarrow \mathcal{A}_{\{\omega\}}(\R)$ is dense if $\omega$ satisfies $(\epsilon)_{0}$ and $(\delta)$. In particular, one also obtains the continuous and dense inclusion $\mathcal{A}'_{\{\omega\}}(\R)\hookrightarrow \mathcal{U}'_{(\omega)}(\C) $.
 \label{dense corollary}
 \end{corollary}

Our next goal is to construct analytic representations of elements of $\mathcal{U}'_{(\omega)}(\R)$ which are globally defined, namely, everywhere outside some closed horizontal strip. The basic idea is to paste together the analytic representations obtained in Proposition \ref{Beurlinglocalcoh} with the aid of a Mittag-Leffler procedure. We define
$$
\mathcal{O}_{(\omega)}^{h} = \bigcup_{ \lambda > 0} \bigcup_{b < h} \bigcap_{R > b} \mathcal{O}_{\omega}^{b,R, \lambda}, \qquad \mathcal{O}_{(\omega)} = \bigcup_{h > 0} \mathcal{O}_{(\omega)}^{h}, \qquad
\mathcal{P}_{(\omega)} = \bigcup_{\lambda > 0} \bigcap_{R > 0} \mathcal{P}_{\omega}^{R, \lambda}.
$$
We use the union and intersection notation to emphasize that we do not topologize the latter spaces. We need the following lemma.
\begin{lemma} \label{runge} Let $0 < L < R$ and let $\omega$ be a weight function satisfying $(\epsilon)_0$ and $(\delta)$. Then, $\mathcal{U}_{(\omega)}(\C)$ is dense in $\mathcal{P}_{\omega}^{R,
\lambda}$ with respect to the norm $| \cdot |^{L,H\lambda}$.
\end{lemma}
\begin{proof}
By Theorem \ref{nontrivial2}, it suffices to show that $\mathcal{A}_{(\omega)}^{R,\infty}$ is dense in $\mathcal{P}_{\omega}^{R,
\lambda}$ with respect to the norm $| \cdot |^{L,H\lambda}$. Let $P \in \mathcal{P}^{R,\lambda}_{\omega}$ and choose $\varphi \in \mathcal{A}_{(\omega)}^{R, \infty}$ with $\varphi(0) = 1$.
Set $P_n(z) = P(z)\varphi(z/n) \in \mathcal{A}_{(\omega)}^{R,\infty}$ for $n \geq 1$. We have
$$
 | P - P_n|^{L, H\lambda} = \sup_{z \in T^{L}} |P(z)(1- \varphi(z/n))|e^{-\omega(H\lambda x)} \leq  A|P|^{ L, \lambda} \sup_{z \in T^{L}} |1- \varphi(z/n)|e^{-\omega(\lambda x)}.
$$
Since $\varphi(z/n) \to 1$, as $n \to \infty$, uniformly on compact subsets of $T^{R}$, this proves the result.
\end{proof}
\begin{proposition}\label{Beurlinglocalcoh2}
Let $\omega$ be a weight function satisfying $(\epsilon)_0$ and $(\delta)$. The sequence
$$
0 \longrightarrow \mathcal{P}_{(\omega)} \longrightarrow \mathcal{O}^{h}_{(\omega)} \xrightarrow{\phantom,\operatorname{bv}\phantom,}  (\mathcal{A}_{(\omega)}^{h})' \longrightarrow 0
$$
is exact. 
\end{proposition}
\begin{proof}
The fact that $\mathcal{P}_{(\omega)} = \ker{\operatorname{bv}}$ is clear from Proposition \ref{Beurlinglocalcoh}. We show that the boundary value mapping is surjective. Let $f  \in (\mathcal{A}_{(\omega)}^{h})'$. By Proposition \ref{Beurlinglocalcoh} there are $0 < b < h$ and $\lambda > 0$ such that for every $n \in \N$ there is $G_n \in \mathcal{O}^{b,b+n+1,\lambda}_\omega$ such that $\operatorname{bv}(G_n) = f$. Corollary \ref{edge} and Lemma \ref{denseinf} yield $G_{n+1} -G_n = P_n \in \mathcal{P}_\omega^{b+n+1,2H^4\lambda}$ and thus, by Lemma \ref{runge}, there is $\varphi_n \in \mathcal{U}_{(\omega)}(\C)$ such that $|P_n - \varphi_n|^{b+n,2H^5\lambda} \leq 2^{-n}$.  Define
$$
F_n(z) = G_n(z) - \sum_{k=0}^{n-1}\varphi_k + \sum_{k=n}^{\infty}(P_k - \varphi_k), \qquad  z \in T^{b,b+n}. 
$$
Then, $F_n \in \ow^{b,b+n, 2H^5\lambda}$, $\operatorname{bv}(F_n) = f$, and $F_{n + 1}(z) = F_n(z)$ for $z \in T^{b,b+n}$. Hence, $F(z) := F_n(z)$ for $z \in T^{b,b+n}$, is a well defined element of $\of^h$ such that $\operatorname{bv}(F) = f$. 
\end{proof}
Summarizing, the following representation theorem should now be clear.
\begin{theorem}\label{Beurlingcoh}
Let $\omega$ be a weight function satisfying $(\epsilon)_0$ and $(\delta)$. The sequence
$$
0 \longrightarrow \mathcal{P}_{(\omega)} \longrightarrow \mathcal{O}_{(\omega)} \xrightarrow{\phantom,\operatorname{bv}\phantom,}  \mathcal{U}'_{(\omega)}(\C) \longrightarrow 0
$$
is exact. 
\end{theorem}
\subsection{Analytic representations of hyperfunctions of $\{\omega\}$-type}\label{analyticrepresentationsroumieu}
We now turn our attention to the representation of $\mathcal{A}'_{\{\omega\}}(\mathbb{R})$. In analogy to the previous subsection, we begin our study with the dual of $\mathcal{A}^h_{\{\omega\}}$. For $0 < h < R$  we set
$$
\mathcal{O}_{\{\omega\}}^{h,R,\lambda} = \varprojlim_{ b \to h^+} \varprojlim_{ L \to R^-} \varprojlim_{\mu \to \lambda^+} \mathcal{O}_{\omega}^{b,L, \mu}, \qquad \mathcal{O}_{\{\omega\}}^{h,R} = \varprojlim_{\lambda \to 0^+} \mathcal{O}_{\{\omega\}}^{h,R,\lambda}, 
$$
and
$$
\mathcal{P}_{\{\omega\}}^{R, \lambda} = \varprojlim_{ L \to R^-} \varprojlim_{\mu \to \lambda^+} \mathcal{P}_{\omega}^{L, \mu}, \qquad \mathcal{P}_{\{\omega\}}^{R} = \varprojlim_{\lambda \to 0^+} \mathcal{P}_{\{\omega\}}^{R, \lambda} .
$$
Lemma \ref{compactanalrepr} implies that, if $\omega$ satisfies $(\delta)$, the spaces $\mathcal{O}_{\{\omega\}}^{h,R}$  and $\mathcal{P}_{\{\omega\}}^{R}$ are $(FS)$-spaces. If $\omega$ satisfies $(\zeta)$ from Lemma \ref{equivalent}, this is also true for $\mathcal{O}_{\{\omega\}}^{h,R, \lambda}$ and $\mathcal{P}_{\{\omega\}}^{R, \lambda}$.  Furthermore, the boundary value mapping
$$
\operatorname{bv}: \mathcal{O}_{\{\omega\}}^{h,R} \rightarrow (\mathcal{A}^h_{\{\omega\}})'
$$
is well defined and continuous. We need to the following density lemma. 
\begin{lemma} \label{denseRoumieu} Let $0 < h < R$ and $0 < \lambda < \mu$. Let $\omega$ be a weight function satisfying $(\epsilon)_0$, $(\delta)$, and $(\zeta)$. Then, $\mathcal{U}_{(\omega)}(\C)$  is dense in $\mathcal{O}_{\{\omega\}}^{h,R,\lambda} \cap \mathcal{P}^{R,\mu}_{\{\omega\}}$ with respect to the topology of $\mathcal{O}_{\{\omega\}}^{h,R,\lambda}$.
\end{lemma}
\begin{proof}
By Theorem \ref{nontrivial2} it is enough to check that $\mathcal{A}_{(\omega)}^{R}$  is dense in $\mathcal{O}_{\{\omega\}}^{h,R,\lambda} \cap \mathcal{P}^{R,\mu}_{\{\omega\}}$ with respect to the topology of $\mathcal{O}_{\{\omega\}}^{h,R,\lambda}$. Let $P \in \mathcal{O}_{\{\omega\}}^{h,R,\lambda} \cap \mathcal{P}^{R,\mu}_{\{\omega\}}$ and choose $\varphi \in \mathcal{A}_{(\omega)}^{R}$ with $\varphi(0) = 1$.
Set $P_n(z) = P(z)\varphi(z/n) \in \mathcal{A}_{(\omega)}^{R}$ for all $n \geq 1$. Let $h < b < L < R$ and $\nu > \lambda$ be arbitrary. For $\lambda < \nu_0 < \nu$ we have
$$
 | P - P_n|^{b, L, \nu} = \sup_{z \in T^{b,L}} |P(z)(1- \varphi(z/n))|e^{-\omega(\nu x)} \leq  |P|^{b, L, \nu_0} \sup_{z \in T^{b,L}} |1- \varphi(z/n)|e^{-(\omega(\nu x) - \omega(\nu_0 x))}.
$$
This shows the lemma because $\omega(\nu t) - \omega(\nu_0 t) \to \infty$, as $t \to \infty$, and $\varphi(z/n) \to 1$, as $n \to \infty$, uniformly on compact subsets of $T^{R}$.
\end{proof}
\begin{proposition}\label{Roumieulocalcoh}
Let  $0 < h < R$ and let $\omega$ satisfy $(\epsilon)_0$ and $(\delta)$. The sequence
$$
0 \longrightarrow \mathcal{P}^R_{\{\omega\}} \longrightarrow \mathcal{O}^{h,R}_{\{\omega\}} \xrightarrow{\phantom,\operatorname{bv}\phantom,}  (\mathcal{A}^h_{\{\omega\}})' \longrightarrow 0
$$
is topologically exact. Moreover, for every bounded set $B \subset (\mathcal{A}^h_{\{\omega\}})'$ there is a bounded set $A \subset \mathcal{O}^{h,R}_{\{\omega\}}$ such that $\operatorname{bv}(A) = B$.
\end{proposition}
\begin{proof}
In view of the open mapping theorem it suffices to show that the sequence is algebraically  exact. Corollary \ref{edge} implies $\mathcal{P}^R_{\{\omega\}} = \ker{\operatorname{bv}}$. We now show that the boundary value mapping is surjective. By Lemma \ref{equivalent} we may assume that $\omega$ satisfies $(\zeta)$. Let $f \in (\mathcal{A}^h_{\{\omega\}})'$ and define
$$
X_n = \{ F \in \mathcal{O}_{\{\omega\}}^{h + (1/n),R,1/n} \, : \, \operatorname{bv}(F) = f  \mbox{ on } \mathcal{U}_{(\omega)}(\C)\}, \qquad n \geq 1.
$$
By Corollary \ref{analquan} $X_n$ is a non-empty closed subspace of $ \mathcal{O}_{\{\omega\}}^{h + (1/n),R,1/n}$ and therefore a complete metrizable topological space (with respect to the relative topology). Consider the projective spectrum $(X_n)_{n \in \Z_+}$ (where the linking mappings are just the inclusion mappings) and let $X$ be its projective limit. It suffices to show that $X$ is non-empty, which would be implied by Lemma \ref{intersection} if we verify that $X_{n+1}$ is dense in $X_n$. Since, by Corollary \ref{edge}, every $F \in X_n$ can be written as $F = G + P$ where $G \in X_{n+1}$ and $P \in \mathcal{O}_{\{\omega\}}^{h + (1/n),R,1/n} \cap \mathcal{P}^{R,4H^4/n}_{\{\omega\}}$, the density of $X_{n+1}$ in $X_{n}$ follows from Lemma \ref{denseRoumieu}. The second part follows from the general fact that for an exact sequence of Fr\'echet spaces 
$$
0 \longrightarrow E \longrightarrow F \xrightarrow{\phantom,T\phantom,} G \longrightarrow 0,
$$
with $E$ an $(FS)$-space, it holds that every bounded set $B \subset G$ can be written as $T(A) = B$ for some bounded set $A \subset F$ \cite[Lemma 26.13]{Meise}.
\end{proof}
We now construct global analytic representations. Set 
$$
\mathcal{O}_{\{\omega\}}^{h} = \varprojlim_{R \to \infty} \mathcal{O}_{\{\omega\}}^{h,R}, \qquad  \mathcal{O}_{\{\omega\}} = \varprojlim_{h \to 0^+} \mathcal{O}_{\{\omega\}}^{h},  \qquad \mathcal{P}_{\{\omega\}} = \varprojlim_{R \to \infty} \mathcal{P}_{\{\omega\}}^{R}.
$$
Note that all of the above spaces are $(FS)$-spaces if $\omega$ satisfies $(\delta)$,  as follows from Lemma \ref{compactanalrepr}.

\begin{proposition}\label{Roumieulocalcoh2}
Let $\omega$ be a weight function satisfying $(\epsilon)_0$ and $(\delta)$. The sequence
$$
0 \longrightarrow \mathcal{P}_{\{\omega\}} \longrightarrow \mathcal{O}^{h}_{\{\omega\}} \xrightarrow{\phantom,\operatorname{bv}\phantom,}  (\mathcal{A}_{\{\omega\}}^{h})' \longrightarrow 0
$$
is topologically exact. Moreover, for every bounded set $B \subset (\mathcal{A}^h_{\{\omega\}})'$ there is a bounded set $A \subset \mathcal{O}^{h}_{\{\omega\}}$ such that $\operatorname{bv}(A) = B$.
\end{proposition}
\begin{proof}
Set $X_n =  \mathcal{P}_{\{\omega\}}^{h+n}$ and $Y_n = \mathcal{O}_{\{\omega\}}^{h,h + n}$ for $n \geq 1$. Consider the following projective sequence of short sequences
\begin{center}
\begin{tikzpicture}
  \matrix (m) [matrix of math nodes, row sep=3em, column sep=3em]
    { 0 & X_1  & Y_1  & (\mathcal{A}^h_{\omega})' & 0 \\
      0 & X_2 & Y_2 & (\mathcal{A}^h_{\omega})'   & 0 \\ 
        & \vdots & \vdots & \vdots &  \\ };
  { [start chain] \chainin (m-1-1);
\chainin (m-1-2);
\chainin (m-1-3);
\chainin (m-1-4) [join={node[above,labeled] {\operatorname{bv}}}];
\chainin (m-1-5); }
  { [start chain] \chainin (m-2-1);
    \chainin (m-2-2);
       { [start branch=X_2] \chainin (m-1-2)
        [join={node[right,labeled] {}}];}
    \chainin (m-2-3);
    { [start branch=Y_2] \chainin (m-1-3)
        [join={node[right,labeled] {}}];}
        
    \chainin (m-2-4) [join={node[above,labeled] {\operatorname{bv}}}];
    { [start branch=Y_2] \chainin (m-1-4)
        [join={node[right,labeled] {}}];}
    \chainin (m-2-5); }
      { [start chain]  
\chainin (m-3-2);
{ [start branch=X_2] \chainin (m-2-2)
        [join={node[right,labeled] {}}];}
 
 }
      { [start chain]  
\chainin (m-3-3);
{ [start branch=X_2] \chainin (m-2-3)
        [join={node[right,labeled] {}}];}
 
 }
      { [start chain]  
\chainin (m-3-4);
{ [start branch=X_2] \chainin (m-2-4)
        [join={node[right,labeled] {}}];}
 
 }
\end{tikzpicture}
\end{center}
By Proposition \ref{Roumieulocalcoh} every horizontal line is a topologically exact sequence. Moreover, by Lemma \ref{runge}, $X_{n+1}$ is dense in $X_n$ for all $n \geq 1$. Hence, the Mittag-Leffler lemma \cite[Lemma 1.3]{Komatsu} yields the topological exactness of the sequence 
$$
 0 \longrightarrow \varprojlim_{n}X_n = \mathcal{P}_{\{\omega\}} \longrightarrow \varprojlim_{n}Y_n = \mathcal{O}^{h}_{\{\omega\}} \xrightarrow{\phantom,\operatorname{bv}\phantom,}  (\mathcal{A}_{\{\omega\}}^{h})' \longrightarrow 0.
$$
The second statement follows from \cite[Lemma 26.13]{Meise} (cf.\ the end of the proof of Proposition \ref{Roumieulocalcoh}).
\end{proof}

We then have,

\begin{theorem}\label{Roumieucoh}
Let $\omega$ be a weight function satisfying $(\epsilon)_0$ and $(\delta)$. The sequence
$$
0 \longrightarrow \mathcal{P}_{\{\omega\}} \longrightarrow \mathcal{O}_{\{\omega\}} \xrightarrow{\phantom,\operatorname{bv}\phantom,}  \mathcal{A}'_{\{\omega\}}(\R) \longrightarrow 0
$$
is topologically exact. Moreover, for every bounded set $B \subset \mathcal{A}'_{\{\omega\}}(\R)$ there is a bounded set $A \subset \mathcal{O}_{\{\omega\}}$ such that $\operatorname{bv}(A) = B$.
\end{theorem}
\begin{proof}  We now set $Y_n =  \mathcal{O}_{\{\omega\}}^{1/n}$ and $Z_n = (\mathcal{A}^{1/n}_{\{\omega\}})'$ for $n \geq 1$. Consider the projective sequence of short sequences
\begin{center}
\begin{tikzpicture}
  \matrix (m) [matrix of math nodes, row sep=3em, column sep=3em]
    { 0 & \mathcal{P}_{\{\omega\}}  & Y_1  & Z_1& 0 \\
      0 & \mathcal{P}_{\{\omega\}} & Y_2 & Z_2  & 0 \\ 
        & \vdots & \vdots & \vdots &  \\ };
  { [start chain] \chainin (m-1-1);
\chainin (m-1-2);
\chainin (m-1-3);
\chainin (m-1-4) [join={node[above,labeled] {\operatorname{bv}}}];
\chainin (m-1-5); }
  { [start chain] \chainin (m-2-1);
    \chainin (m-2-2);
       { [start branch=X_2] \chainin (m-1-2)
        [join={node[right,labeled] {}}];}
    \chainin (m-2-3);
    { [start branch=Y_2] \chainin (m-1-3)
        [join={node[right,labeled] {}}];}
        
    \chainin (m-2-4) [join={node[above,labeled] {\operatorname{bv}}}];
    { [start branch=Y_2] \chainin (m-1-4)
        [join={node[right,labeled] {}}];}
    \chainin (m-2-5); }
      { [start chain]  
\chainin (m-3-2);
{ [start branch=X_2] \chainin (m-2-2)
        [join={node[right,labeled] {}}];}
 
 }
      { [start chain]  
\chainin (m-3-3);
{ [start branch=X_2] \chainin (m-2-3)
        [join={node[right,labeled] {}}];}
 
 }
      { [start chain]  
\chainin (m-3-4);
{ [start branch=X_2] \chainin (m-2-4)
        [join={node[right,labeled] {}}];}
 
 }
\end{tikzpicture}
\end{center}
By Proposition \ref{Roumieulocalcoh2} every horizontal line is a topologically exact sequence. Hence, the Mittag-Leffler lemma \cite[Lemma 1.3]{Komatsu} yields the topologically exact sequence 
$$
 0 \longrightarrow \mathcal{P}_{\{\omega\}} \longrightarrow \varprojlim_{n}Y_n = \mathcal{O}_{\{\omega\}} \xrightarrow{\phantom,\operatorname{bv}\phantom,}  \varprojlim_{n}Z_n = \mathcal{A}'_{\{\omega\}}(\R) \longrightarrow 0.
$$
For the second statement, we apply again \cite[Lemma 26.13]{Meise}.
\end{proof}
\section{Support}\label{sect-support}
Silva  introduced in \cite{Silva} a useful notion of real support for ultrahyperfunctions. We extend such considerations to ultrahyperfunctions and hyperfunctions of fast growth via the analytic continuation properties of their analytic representations. We will use these ideas to establish a support separation theorem, which in particular gives that every (ultra)hyperfunction of fast growth can be written as the sum of two (ultra)hyperfunctions of fast growth having support in the positive and negative half-axis, respectively. Based upon this result, we shall study in Section \ref{section laplace} analytic representations of ultradistributions of class $\omega$ of exponential type via Laplace transforms.
 
\subsection{Real support of ultrahyperfunctions of $(\omega)$-type}
Let $\overline{\R} = \R \cup \{-\infty, \infty\}$ be the extended real line endowed with its usual topology (two-point compactification of $\R$).  For $0 < b < R$ and $A \subseteq \overline{\R}$ we set $T^R(A) = (A \cap \R) + i(-R,R)$ and $T^{b,R}(A) = T^{b,R} \cup T^R(A)$. Let $\omega$ be a weight function. Given $K \Subset \overline{\R}$ a proper closed set with non-empty interior we denote by  $\mathcal{O}_{\omega}^{b,R}(K)$ the $(B)$-space of functions $F \in \mathcal{O}(T^{b,R}(\operatorname{int} K))$ that satisfy
$$
|F|^{b,R}_{\omega,K} :=  \sup \{|F(z)| e^{-\omega(x)} : \, z \in T^{b,R}(\operatorname{int} K) \}< \infty.
$$
For an open subset $\Omega$ of $\overline{\R}$, we then define
$$
 \mathcal{O}_{\omega}^{h,R}(\Omega) = \varprojlim_{K \Subset \Omega} \varprojlim_{b \to h^+} \varprojlim_{L \to R^-}  \mathcal{O}_{\omega}^{b,L}(K).
$$
We also write $\mathcal{O}^{b,R}_{\omega_\lambda}(K) = \mathcal{O}^{b,R, \lambda}_{\omega}(K)$, $\mathcal{O}^{h,R}_{\omega_\lambda}(\Omega) = \mathcal{O}^{h,R, \lambda}_{\omega}(\Omega)$, and
$$| F |^{b,R}_{\omega_\lambda,K} = | F |^{b,R,\lambda}_{\omega,K}= | F |^{b,R,\lambda}_K, \qquad F \in \mathcal{O}^{b,R, \lambda}_{\omega}(K).$$
Furthermore, we set
$$
\mathcal{O}^{h,R}_{(\omega)}(\Omega) = \bigcup_{\lambda >0} \mathcal{O}_{\omega}^{h,R, \lambda}(\Omega), \qquad  \mathcal{O}_{(\omega)}(\Omega+i\mathbb{R}) = \bigcup_{\lambda,h >0} \bigcap_{R >h} \mathcal{O}_{\omega}^{h,R, \lambda}(\Omega).
$$
Suppose that $\omega$ satisfies $(\delta)$ and $(\epsilon)_0$. Let $f \in \mathcal{U}'_{(\omega)}(\C)$ and $\Omega\subseteq\overline{\R}$ be open. We say that $f$ \emph{vanishes on} $\Omega$ if there is $F \in \mathcal{O}_{(\omega)}(\Omega+i\mathbb{R})$ such that $\operatorname{bv}(F) = f$. Theorem \ref{Beurlingcoh} implies that in such a case this property holds for all analytic representations of $f$ (in $\mathcal{O}_{(\omega)}$). Moreover, the proof of Proposition \ref{Beurlinglocalcoh2} shows that for $f$ to vanish on $\Omega$ it suffices that there is $F \in \mathcal{O}^{h,R}_{(\omega)}(\Omega)$, for some $0 < h < R$, such that $\operatorname{bv}(F) = f$. 

We define the \emph{real support of} $f \in \mathcal{U}'_{(\omega)}(\C)$, denoted by $\operatorname*{supp_{\overline{\R}}} f$, as the complement of the largest open  subset of $\overline{\R}$ on which $f$ vanishes. Given a closed subset $K\subseteq\overline{\R}$, we write
$$
\mathcal{U}'_{(\omega)}[K + i\R]= \{f\in\mathcal{U}'_{(\omega)}(\C)\: : \operatorname*{supp_{\overline{\R}}} f\subseteq K \},$$
the subspace of ultrahyperfunctions of $(\omega)$-type with real support in $K$.  When $I$ is a closed interval of $\overline{\R}$,  $f \in \mathcal{U}'_{(\omega)}[I + i\R]$, and $F \in \mathcal{O}^{h,R}_{(\omega)}(\overline{\R} \backslash I)$ is such that $\operatorname{bv}(F) = f$, Cauchy's integral theorem yields
$$
\langle f, \varphi \rangle = -\int_{\Gamma^{b}(J)}F(z) \varphi(z)dz, \qquad  \varphi \in \mathcal{U}_{(\omega)}(\C),
$$ 
where $h < b < R $, $J$ an interval in $\overline{\R}$ such that $I \Subset J$, and $\Gamma^b(J)$ the boundary of $T^b(J)$. More generally, if $K$ is a proper closed subset of $\overline{\R}$ and $J_{1},J_{2},\dots, J_{n}$ is a finite covering of $K$ by open intervals of the extended real line such that their finite end points do not belong to $K$, and $f=\operatorname{bv}(F) \in \mathcal{U}'_{(\omega)}[K + i\R]$ with $F \in \mathcal{O}^{h,R}_{(\omega)}(\overline{\R} \backslash K)$, we have the representation
\begin{equation}
\label{eqrepultracontour}
\langle f, \varphi \rangle = -\left(\int_{\Gamma^{b}(J_{1})}+\int_{\Gamma^{b}(J_{2})}+\dots+\int_{\Gamma^{b}(J_{n})}\right)F(z) \varphi(z)dz,
\end{equation}
for each $\varphi \in \mathcal{U}_{(\omega)}(\C)$.

The space $\mathcal{U}'_{(\omega)}[K + i\R]$ consists of analytic functionals if $K$ is a compact subset of $\mathbb{R}$. Indeed, if $S\subseteq\mathbb{C}$ is a closed set, we write
$$
\mathcal{O}[S] = \varinjlim_{ S\Subset V} \mathcal{O}(V).
$$
The Silva-K\"othe-Grothendieck theorem \cite{Morimoto2} now yields 
$$
\mathcal{U}'_{(\omega)}[K + i\R]  = \bigcup_{R > 0} \mathcal{O}'[K + i[-R,R]]=\mathcal{O}'[K+i\mathbb{R}],
$$
which is in fact a result due to Silva \cite{Silva}. 

Our next goal is to show a support separation theorem for $\mathcal{U}'_{(\omega)}(\C)$. For $a,b \in \R$, we employ the special notations
$$
\mathcal{U}'_{(\omega),a+} = \mathcal{U}'_{(\omega)}[[a,\infty] + i\R], \qquad \mathcal{U}'_{(\omega),b-} = \mathcal{U}'_{(\omega)}[[-\infty,b] + i\R].
$$
\begin{theorem}\label{splittingBeurling}
Let $-\infty < a \leq b < \infty$ and let $\omega$ be a weight function satisfying  $(\delta)$ and $(\epsilon)_0$. The sequence 
$$
0 \longrightarrow \mathcal{U}'_{(\omega)}[[a,b] + i\R] \longrightarrow \mathcal{U}'_{(\omega),a+}\oplus \mathcal{U}'_{(\omega),b-} \xrightarrow{\phantom-\lambda\phantom-}  \mathcal{U}'_{(\omega)}(\C) \longrightarrow 0
$$
is exact, where $\lambda((f_1,f_2)) = f_1 - f_2$.
\end{theorem}
We need some extra notions to prove Theorem \ref{splittingBeurling}. Let $\Omega\subseteq \overline{\R}$ be open and define $\mathcal{A}_{\omega}^{h}(\Omega)$ as the $(B)$-space of analytic functions $\varphi \in \mathcal{O}(T^h(\Omega))$ such that
$$
\| \varphi \|^h_{\omega, \Omega} := \sup_{z \in T^h(\Omega)}|\varphi(z)|e^{\omega(x)} < \infty. 
$$
Set $\mathcal{A}^h_{\omega_\lambda}(\Omega) =  \mathcal{A}_{\omega}^{h,\lambda}(\Omega)$ and
$$\| \varphi \|^h_{\omega_\lambda,\Omega} = \| \varphi \|^{h, \lambda}_{\omega,\Omega} = \| \varphi \|^{h,\lambda}_\Omega, \qquad \varphi \in \mathcal{A}_{\omega}^{h,\lambda}(\Omega).$$
We have the following refinement of Proposition \ref{analyticreprgeneral}.
\begin{proposition}\label{analyticreprgeneral2}
Let $0 < k < b < h < R$, let $U \Subset V \Subset \Omega$ be open subsets of $\overline{\R}$, and let $\omega$, $\sigma$, and $\kappa$ be three weight functions satisfying \eqref{condition0} and \eqref{condition}. Furthermore, suppose there is $P \in \mathcal{O}(T^R)$ satisfying
$$
C_1e^{\omega(x)} \leq |P(z)| \leq C_2e^{\sigma(x)}, \qquad  z \in T^R,
$$
for some $C_1, C_2 >0$. Then, for every $f \in (\mathcal{A}^k_\omega(U))'$ there is  $F \in \mathcal{O}_\sigma^{b,R}(\overline{\R}\backslash V)$ such that $\operatorname{bv}(F) = f$ on $\mathcal{A}_{\kappa}^h(\Omega)$.
\end{proposition}
In view of Proposition \ref{funda}, we have the following corollary.
\begin{corollary}\label{analquanlocal}
Let $0 < k < b < h < R$, let $U \Subset V \Subset \Omega$ be open subsets of $\overline{\R}$, and let $\omega$ be a weight function satisfying $(\delta)$ and $(\epsilon)_0$.  For every $f \in (\mathcal{A}^{k,\lambda}_\omega(U))'$ there is  $F \in \mathcal{O}_\omega^{b,R,2H^2\lambda}(\overline{\R}\backslash V)$ such that $\operatorname{bv}(F) = f$ on $\mathcal{A}^{h,2H^3\lambda}_\omega(\Omega)$.
\end{corollary}
The proof of Theorem \ref{splittingBeurling} is based on the following functional analytic method for checking the surjectivity of linear continuous mappings.
\begin{lemma}\label{checksurj}
Let $E_1$, $E_2$, and $E$ be semi-reflexive $(DF)$-spaces and let $\rho_j: E \rightarrow E_j$, $j =1,2$, be continuous linear mappings. Then, 
\[ \iota: E'_1 \times E'_2 \rightarrow E': (x'_1, x'_2) \rightarrow \rho^t_1(x'_1) + \rho^t_2(x'_2) \]
is surjective if and only if $\rho : E \rightarrow E_1 \times E_2: x \rightarrow (\rho_1(x), \rho_2(x))$ is injective and has closed range.
\end{lemma}
\begin{proof}
Since $E_1$, $E_2$, and $E$  are semi-reflexive we may identify the transpose of $\iota$ with $\rho$. A continuous linear mapping between two Fr\'echet spaces is surjective if and only if its transpose is injective and has weakly closed range \cite[Thm.\ 37.3]{Treves}. The result now follows from the fact that the closed convex sets of a Hausdorff locally convex space are the same for all topologies compatible with the dual pairing. 
\end{proof}
\begin{proof}[Proof of Theorem \ref{splittingBeurling}.]
Theorem \ref{Beurlingcoh} implies that
$$
 \mathcal{U}'_{(\omega)}[[a,b] + i\R] = \mathcal{U}'_{(\omega),a+} \cap \mathcal{U}'_{(\omega),b-}.
$$
It remains to show that $\lambda$ is surjective. By Lemma \ref{equivalent} we may assume that $\omega$ satisfies $(\zeta)$. Let $f \in \mathcal{U}'_{(\omega)}(\C)$, we show that there are $f_1 \in \mathcal{U}'_{(\omega),a+}$ and $f_2 \in  \mathcal{U}'_{(\omega),b-}$ such that $f = f_1 + f_2$. Choose $h, \lambda > 0$ such that $f$ can be extended to an element of $(\mathcal{A}^{h,\lambda}_\omega)'$ (we also write $f$ for this extension). Define 
$$
X^{h,\lambda} = \varinjlim_{\mu \to \lambda^+}\varinjlim_{k \to h^+} \mathcal{A}^{k,\mu}_\omega,
$$
and
$$
X_{a+}^{h,\lambda} =\varinjlim_{\mu \to \lambda^+} \varinjlim_{k \to h^+}\varinjlim_{\varepsilon \to 0^+} \mathcal{A}^{k,\mu}_\omega((a-\varepsilon, \infty]), \qquad X_{b-}^{h,\lambda} =\varinjlim_{\mu \to \lambda^+} \varinjlim_{k \to h^+}\varinjlim_{\varepsilon \to 0^+} \mathcal{A}^{k,\mu}_\omega([-\infty,b+\varepsilon)).
$$
Condition $(\zeta)$ implies that the three above spaces are $(DFS)$-spaces. Observe that $f \in (X^{h,\lambda})'$. Let $g \in (X_{a+}^{h,\lambda})'$ and choose $R > h$. By Corollary \ref{analquanlocal}, we have that for every $\varepsilon > 0$ and $k > h$ there is $G_{\varepsilon,k} = G \in  \mathcal{O}_\omega^{k,R,4H^2\lambda}([-\infty,a-\varepsilon])$ such that $\operatorname{bv}(G) = g$. By a similar Mittag-Leffler procedure as in the proof of Proposition \ref{Beurlinglocalcoh2}, one can now show that there is $G \in  \mathcal{O}_\omega^{h,R,8H^7\lambda}([-\infty,a))$ such that $\operatorname{bv}(G) = g$. Likewise, it holds that for every $g \in (X_{b-}^{h,\lambda})'$ there is $G \in  \mathcal{O}_\omega^{h,R,8H^7\lambda}((b,\infty])$ such that $\operatorname{bv}(G) = g$. Therefore, it suffices to show that the mapping
$$
 (X_{a+}^{h,\lambda})'  \times  (X_{b-}^{h,\lambda})' \rightarrow  (X^{h,\lambda})': (f_1, f_2) \rightarrow f_1 + f_2
$$
is surjective; but the latter is a consequence of Lemma \ref{checksurj}.
\end{proof}
\subsection{Support of hyperfunctions of $\{\omega\}$-type} We now define the support of hyperfunctions of  $\{\omega\}$-type. Given a compact subset $K\subset\mathbb{R}$, we write $\mathcal{A}[K]=\mathcal{O}[K]$ for the space of germs of analytic functions on $K$, so that $\mathcal{A}'[K]$ is the space of analytic functionals on $K$.

Let $\omega$ be a weight function satisfying $(\delta)$ and $(\epsilon)_0$. Given a proper open subset $\Omega$ of $\overline{\R}$, we define
$$
\mathcal{O}_{\{\omega\}}(\Omega) = \varprojlim_{\lambda \to 0^+} \varprojlim_{h \to 0^+} \varprojlim_{R \to \infty} \mathcal{O}_{\omega}^{h,R, \lambda}(\Omega),
$$
an $(FS)$-space. We say that $f \in \mathcal{A}'_{\{\omega\}}(\R)$ vanishes on $\Omega$ if there is $F \in \mathcal{O}_{\{\omega\}}(\Omega)$ such that $\operatorname{bv}(F) = f$. In view of Theorem \ref{Roumieucoh}, this definition is independent of the chosen analytic representation. We can therefore define the support of $f$, denoted by $\operatorname*{supp}f$, as the complement of the largest open set where $f$ vanishes. 

Given a closed subset $K$ in $\overline{\R}$, we set 
$$\mathcal{A}'_{\{\omega\}}[K]=\{f\in \mathcal{A}'_{\{\omega\}}(\mathbb{R}):\:  \operatorname*{supp}f \subseteq K\},$$ 
and for $a,b \in \R$, we write
$$
\mathcal{A}'_{\{\omega\},a+} = \mathcal{A}'_{\{\omega\}}[[a,\infty]], \qquad \mathcal{A}'_{\{\omega\},b-} = \mathcal{A}'_{\{\omega\}}[[-\infty,b]].
$$
If $f=\operatorname{bv}(F) \in \mathcal{A}'_{\{\omega\}}[K]$ with $F \in \mathcal{O}_{\{\omega\}}(\overline{\R} \backslash K)$ and $J_{1},J_{2},\dots, J_{n}$ is a finite covering of $K$ by open intervals of the extended real line (such that their finite end points do not belong to $K$), Cauchy's integral theorem also gives the contour integral representation (\ref{eqrepultracontour}), where $b > 0$ depends on $\varphi$ (concretely, for $\varphi\in\mathcal{A}^b_{\{\omega\}}$). In particular, for a compact set in $K$ in $\R$, 
$$
\mathcal{A}'_{\{\omega\}}[K] = \mathcal{A}'[K],
$$
as follows from the Silva-K\"othe-Grothendieck theorem. In case $K$ is unbounded, we can also represent $\mathcal{A}'_{\{\omega\}}[K] $ as a dual space. We define the (DFS)-space
$$
\mathcal{A}_{\{\omega\}}[K] = \varinjlim_{\lambda \to 0^+} \varinjlim_{h \to 0^+}  \varinjlim_{K \Subset \Omega}\mathcal{A}^{h,\lambda}_{\omega}(\Omega).
$$
Corollary \ref{analquanlocal} and the same method as in Section \ref{analyticrepresentationsroumieu} yield:
\begin{theorem}\label{dualcharacterization}
Let $K$ be a closed subset of $\overline{\R}$ and let $\omega$ be a weight function satisfying $(\delta)$ and $(\epsilon)_0$. Then, $(\mathcal{A}_{\{\omega\}}[K])' = \mathcal{A}'_{\{\omega\}}[K]$. Furthermore, the sequence
$$
0 \longrightarrow \mathcal{P}_{\{\omega\}} \longrightarrow \mathcal{O}_{\{\omega\}}(\overline{\R}\setminus K)  \xrightarrow{\phantom-\operatorname*{bv}\phantom-}  (\mathcal{A}_{\{\omega\}}[K])' \longrightarrow 0
$$
is topologically exact. 
\end{theorem}
We also have the ensuing support separation theorem.
\begin{theorem}\label{splittingRoumieu}
Let $-\infty < a \leq b < \infty$ and let $\omega$ be a weight function satisfying  $(\delta)$ and $(\epsilon)_0$. The sequence 
$$
0 \longrightarrow \mathcal{A}'_{\{\omega\}}[[a,b]] \longrightarrow \mathcal{A}'_{\{\omega\},a+} \oplus \mathcal{A}'_{\{\omega\},b-} \xrightarrow{\phantom-\lambda\phantom-}  \mathcal{A}'_{\{\omega\}}(\R) \longrightarrow 0
$$
is topologically exact, where $\lambda((f_1,f_2)) = f_1 - f_2$.
\end{theorem}
\begin{proof} Theorem \ref{Roumieucoh} implies that
$$
 \mathcal{A}'_{\{\omega\}}[[a,b]] = \mathcal{A}'_{\{\omega\},a+} \cap \mathcal{A}'_{\{\omega\},b-}.
$$
The surjectivity of $\lambda$ follows from Lemma \ref{checksurj} and Theorem \ref{dualcharacterization}. 
\end{proof}

\section{Spaces of (ultra)hyperfunctions defined via subadditive weight functions} \label{section subadditive}
In this section we briefly indicate how the results from Sections \ref{testfunctions}--\ref{sect-support} can be extended to include spaces defined in terms of subadditive weight functions. Since the theory and methods are completely analogous to those already developed, we shall omit all proofs. These ideas will be applied in Section \ref{section laplace} to the study of analytic representations of ultradistributions of exponential type and Laplace transforms.

\subsection{Subadditive weight functions} \label{subsection subadditive} We collect here a number of properties of the weight functions that we shall employ in the rest of the paper. Let $\omega$ be a weight function (not necessarily satisfying \eqref{fasterthanlog}) such that $\omega(0) = 0$ (this will be assumed from now on). We are interested in the following conditions \cite{Bjorck}:
\begin{itemize}
\item[$(\alpha)$] $\omega(t_1 +t_2) \leq \omega(t_1) + \omega(t_2)$, $ t_1,t_2 \geq 0$,
\item[$(\gamma)$] $\omega(t) \geq c\log(1+t) + a$, for some $a \in \R$ and $c > 0$.
\end{itemize}

In the sequel we shall refer to condition \eqref{fasterthanlog} as $(\gamma)_0$. We set ${}_\lambda \omega(t) = \lambda\omega(t)$. Let $\sigma$ be another weight function, $\omega$ and $\sigma$ are said to be \emph{$\ast$-equivalent}, denoted by $\omega \asymp \sigma$, if $\omega(t)=O(\sigma(t))$ and $\sigma(t)=O(\omega(t))$ (as $t\to\infty$), or equivalently, 
$$
 {}_{\lambda_1} \omega(t) - C_1 \leq \sigma(t) \leq  {}_{\lambda_2} \omega(t) + C_2, \qquad t \geq 0,
$$
 for some $\lambda_1, \lambda_2, C_1, C_2 > 0$. If $\omega$ and $\sigma$ both satisfy $(\alpha)$ and $(\delta)$, then $\omega \asymp \sigma$ if and only if $\omega \sim \sigma$. 
 We point out that subadditivity, that is, condition $(\alpha)$,  always yields \cite[p.~240]{korevaarbook} the existence of the limit 
 $$
 \lim_{t\to\infty}\frac{\omega(t)}{t}.
 $$
Consequently, we either have $\omega(t) \asymp t$ or $\omega(t) = o(t)$.

 We shall need the following result of Petzsche and Vogt which allows one to replace a weight function by one enjoying better regularity properties.
\begin{lemma} \cite[Prop.\ 1.2]{Petzsche} \label{equivalentweightsub}
Let $\omega$ be a weight function satisfying $(\alpha)$ and $(\gamma)$. Then, there is another weight function $\sigma$ with $\omega \asymp \sigma$ such that either $\sigma(t) = t$ or $\sigma$ satisfies the following conditions:
\begin{itemize}
\item[$(i)$] $\displaystyle \sigma \in C^\infty(0, \infty)$ and $\lim_{t \to 0^+}\sigma'(t) = \infty$,
\item[$(ii)$] $\sigma$ is strictly concave,
\item[$(iii)$] $\displaystyle \lim_{t \to \infty}\sigma'(t) = 0$,
\item[$(iv)$] $\displaystyle \liminf_{t \to \infty}|\sigma''(t)|t^2 > 0$.
\end{itemize}
If $\omega$ satisfies $(\gamma)_0$, then condition $(iv)$ may be replaced by the stronger property
\begin{itemize}
\item[$(iv)'$] $\displaystyle \lim_{t \to \infty}|\sigma''(t)|t^2 = \infty$.
\end{itemize}
\end{lemma} 

The \emph{Young conjugate} of $\omega$ is defined as
$$
\omega^*(s)= \sup_{t \geq 0} (\omega(t) - ts),  \qquad s > 0,
$$
and is a convex non-increasing function. We set $\omega^*(s) = \omega^*(|s|)$ for $s \in \R$, $s \neq 0$.
If $\omega(t) = o(t)$, then $\omega^*(s) < \infty$ for all $s > 0$. Clearly, we have 
$$
({}_\lambda\omega)^*(s)= \lambda\omega^*\left(\frac{s}{\lambda}\right),  \qquad s > 0.
$$
\subsection{Test function spaces} Given a weight function $\omega$, we set $\mathcal{A}^h_{{}_\lambda \omega} = \mathcal{A}_{\boldsymbol{\omega}}^{h,\lambda}$ and
$$
\| \varphi \|^h_{{}_\lambda \omega} = \| \varphi \|^{h, \lambda}_{\boldsymbol{\omega}}, \qquad \varphi \in \mathcal{A}_{\boldsymbol{\omega}}^{h,\lambda}.
$$
Our basic spaces of entire and analytic functions are now defined as
$$
\mathcal{U}_{\boldsymbol{(\omega)}}(\C) =  \varprojlim_{ \lambda \to \infty}\varprojlim_{ h \to \infty}  \mathcal{A}_{\boldsymbol{\omega}}^{h,\lambda}, \qquad \mathcal{A}_{\boldsymbol{\{\omega\}}}(\R) =  \varinjlim_{\lambda \to 0^+}\varinjlim_{ h \to 0^+} \mathcal{A}_{\boldsymbol{\omega}}^{h,\lambda}.
$$
If $\omega(t) \rightarrow \infty$, as $t \to \infty$, then Lemma \ref{compact} yields that $\mathcal{U}_{\boldsymbol{(\omega)}}(\C)$ is an $(FS)$-space while $\mathcal{A}_{\boldsymbol{\{\omega\}}}(\R)$ is a (DFS)-space.
Let $\sigma$ be another weight such that $\omega \asymp \sigma$, then $\mathcal{U}_{\boldsymbol{(\omega)}}(\C) = \mathcal{U}_{\boldsymbol{(\sigma)}}(\C)$ and $\mathcal{A}_{\boldsymbol{\{\omega\}}}(\R)= \mathcal{A}_{\boldsymbol{\{\sigma\}}}(\R)$, topologically. If $\omega$ satisfies both conditions $(\alpha)$ and $(\delta)$, then $\mathcal{U}_{\boldsymbol{(\omega)}}(\C) = \mathcal{U}_{(\omega)}(\C)$ and $\mathcal{A}_{\boldsymbol{\{\omega\}}}(\R)= \mathcal{A}_{{\{\omega\}}}(\R)$, topologically. The elements of the dual spaces $\mathcal{U}'_{\boldsymbol{(\omega)}}(\C)$ and $\mathcal{A}'_{\boldsymbol{\{\omega\}}}(\R)$ are called \emph{ultrahyperfunctions of $\boldsymbol{(\omega)}$-type} and \emph{hyperfunctions of $\boldsymbol{\{\omega\}}$-type}, respectively. Observe that $\mathcal{U}_{\boldsymbol{(\log(1+t))}}(\C) = \mathcal{U}(\C)$ is the test function space for the Silva space of tempered ultrahyperfunctions \cite{Hoskins,hasumi,Silva} (called tempered ultradistributions there). If $\omega$ satisfies $(\alpha)$  and $(\gamma)$ ($(\gamma)_{0}$), we have the continuous and dense inclusions $\mathcal{S}_{(p!)}^{(p!)}=\mathcal{U}_{(t)}(\C) =\mathcal{U}_{\boldsymbol{(t)}}(\C)\rightarrow \mathcal{U}_{\boldsymbol{(\omega)}}(\C)$ and $\mathcal{S}^{\{p!\}}_{\{p!\}}= \mathcal{A}_{\{t\}}(\R) = \mathcal{A}_{\boldsymbol{\{t\}}}(\R) \rightarrow \mathcal{A}_{\boldsymbol{\{\omega\}}}(\R)$ (density follows from Theorem \ref{cohsub} below). 

We mention that if a (log-convex) weight sequence $M_p$ satisfies $(M.2)$ and
\begin{itemize}
\item [$(M.4)$] $\displaystyle \frac{M_{p}}{p!}\subseteq \left(\frac{M_p}{p!}\right)^{c}$,
\end{itemize}
Petzsche and Vogt have shown \cite[Sect.\ 5]{Petzsche} that there is a weight function $\omega$ satisfying $(\alpha)$, $(\gamma)_{0}$, and $\omega \asymp M$; under these circumstances, we thus have
$\mathcal{U}_{(M_p)}(\C) = \mathcal{U}_{\boldsymbol{(\omega)}}(\C)$ and $\mathcal{A}_{\{M_p\}}(\R) = \mathcal{A}_{\boldsymbol{\{\omega\}}}(\R)$, topologically.
Furthermore, they proved, under $(M.2)$, that $(M.4)$ is equivalent to the so-called Rudin condition:
\begin{itemize}
\item [$(M.4)''$] $\displaystyle \max_{q\leq p}\left(\frac{M_{q}}{q!}\right)^{\frac{1}{q}}\leq A\left(\frac{M_{p}}{p!}\right)^{\frac{1}{p}},$ $p\in\mathbb{N}$, for some $A>0$.
\end{itemize}
Observe that strong non-quasianalyticity (i.e., Komatsu's condition $(M.3)$ \cite{Komatsu}) automatically yields $(M.4)$, as shown by Petzsche \cite[Prop.\ 1.1]{Petzsche2}.  

\begin{theorem}\label{non-trivial-outside}
Let $\omega$ be a weight function satisfying $(\gamma)_0$. Then, $\mathcal{U}_{\boldsymbol{(\omega)}}(\C)$ ($\mathcal{A}_{\boldsymbol{\{\omega\}}}(\R)$) is non-trivial if and only if $\omega$ satisfies $(\epsilon)_0$ ($(\epsilon)_\infty$).
\end{theorem}
\begin{proof}
The conditions are necessary by Proposition \ref{mandel}. In the Roumieu case sufficiency is a consequence of Proposition \ref{funda}. For the Beurling case we notice that, by Lemma \ref{majorizationweight}, we may assume that $\omega$ satisfies $(\delta)$. Hence, $\mathcal{U}_{(\omega)}(\C)  \subseteq \mathcal{U}_{\boldsymbol{(\omega)}}(\C)$ and the result follows from Theorem \ref{nontrivial2}.
\end{proof}
\begin{remark}\label{remark-non-triv}
In \cite{Langenbruch} Langenbruch studied the space $\mathcal{A}_{\boldsymbol{\{\omega\}}}(\R)$ for weight functions $\omega$ satisfying $(\gamma)_0$ and
\begin{equation}
\limsup_{t \to \infty} \frac{\omega(t+1)}{\omega(t)} < \infty.
\label{cond-L}
\end{equation}
He showed that the space $\mathcal{A}_{\boldsymbol{\{\omega\}}}(\R)$ is (tamely) isomorphic to the strong dual of a finite type power series space \cite[Thm.\ 4.4]{Langenbruch} and thus, in particular, is non-trivial. Under the extra assumption \eqref{cond-L}, Theorem \ref{non-trivial-outside} in the Roumieu case is therefore due to Langenbruch. We mention that one can readily show that condition \eqref{cond-L} implies $(\epsilon)_\infty$ (without passing through Theorem \ref{non-trivial-outside}), whereas there exist weight functions satisfying $(\epsilon)_\infty$ and $(\gamma)_0$ but not \eqref{cond-L}.
\end{remark}

Finally, let us discuss the Fourier transform on our test function spaces. We set $\mathcal{K}^h_{1,{}_\lambda \omega}(\R) = \mathcal{K}^{h,\lambda}_{1,\boldsymbol{\omega}}(\R)$ and
$$
\rho^h_{{}_\lambda\omega}(\psi) = \rho^{h,\lambda}_{\boldsymbol{\omega}}(\psi), \qquad \psi \in \mathcal{K}^{h,\lambda}_{1,\boldsymbol{\omega}}(\R);
$$
define then
$$
\mathcal{K}_{1,\boldsymbol{(\omega)}}(\R)= \varprojlim_{\lambda \to \infty} \varprojlim_{h \to \infty}\mathcal{K}^{h,\lambda}_{1,\boldsymbol{\omega}}(\R), \qquad \mathcal{K}_{1,\boldsymbol{\{\omega\}}}(\R)= \varinjlim_{\lambda \to 0^+} \varinjlim_{h \to 0^+}\mathcal{K}^{h,\lambda}_{1,\boldsymbol{\omega}}(\R).  
$$
In analogy to the terminology from Section \ref{testfunctions}, we call $\mathcal{K}'_{1,\boldsymbol{(\omega)}}(\R)$ and $\mathcal{K}'_{1,\boldsymbol{\{\omega\}}}(\R)$ the spaces of \emph{ultradistributions of class $\boldsymbol{(\omega)}$ of exponential type} and \emph{ultradistributions of class $\boldsymbol{\{\omega\}}$ of infra-exponential type}, respectively. If $\omega$ satisfies $(\alpha)$, $(\gamma)$, and is non-quasianalytic, then we have the continuous and dense inclusions $\mathcal{D}_{(\omega)}(\R) \hookrightarrow \mathcal{K}_{1,\boldsymbol{(\omega)}}(\R)\hookrightarrow\mathcal{K}_{1,\boldsymbol{(\omega)}}'(\R)\hookrightarrow \mathcal{D}'_{(\omega)}(\R) $ (see \cite{Bjorck} for the definition of the Beurling-Bj\"orck space $\mathcal{D}_{(\omega)}(\R)$). In particular, $\mathcal{K}_{1,\boldsymbol{(\log(1+t))}}(\R) = \mathcal{K}_1(\R)$ is the space of exponentially rapidly decreasing smooth functions \cite{Hoskins} and we have that $\mathcal{D}(\R) \hookrightarrow \mathcal{K}_1(\R)\hookrightarrow \mathcal{K}_{1}'(\mathbb{R})\hookrightarrow\mathcal{D}'(\R)$.
\begin{proposition}\label{fouriersub}
Let $\omega$ be a weight function satisfying $(\alpha)$ and $(\gamma)$ ($(\gamma)_0$). The Fourier transform is a topological isomorphism from $\mathcal{U}_{\boldsymbol{(\omega)}}(\C)$ ($\mathcal{A}_{\boldsymbol{\{\omega\}}}(\R)$) onto $\mathcal{K}_{1,\boldsymbol{(\omega)}}(\R)$ ($\mathcal{K}_{1,\boldsymbol{\{\omega\}}}(\R)$).
\end{proposition} 
\begin{proof}
This can be shown in the same way as Proposition \ref{fourier}.
\end{proof}
The Fourier transform from $\mathcal{K}'_{1,\boldsymbol{(\omega)}}(\R)$ ($\mathcal{K}'_{1,\boldsymbol{\{\omega\}}}(\R)$) onto $\mathcal{U}'_{\boldsymbol{(\omega)}}(\C)$ ($\mathcal{A}'_{\boldsymbol{\{\omega\}}}(\R)$) is therefore well defined via duality.

\subsection{Boundary values}\label{subreview} Given $0 < b < R$, we set $\mathcal{O}^{b,R}_{{}_\lambda\omega} = \mathcal{O}^{b,R, \lambda}_{\boldsymbol{\omega}}$ and $\mathcal{P}^{R}_{{}_\lambda\omega} = \mathcal{P}^{R, \lambda}_{\boldsymbol{\omega}}$. Furthermore, we define
$$
\mathcal{O}_{\boldsymbol{(\omega)}} = \bigcup_{ \lambda,h > 0 } \bigcap_{R > h} \mathcal{O}^{h,R, \lambda}_{\boldsymbol{\omega}},  \qquad \mathcal{P}_{\boldsymbol{(\omega)}} = \bigcup_{\lambda > 0} \bigcap_{R > 0} \mathcal{P}^{R, \lambda}_{\boldsymbol{\omega}},
$$
and
$$
\mathcal{O}_{\boldsymbol{\{\omega\}}} = \varprojlim_{\lambda \to 0^+} \varprojlim_{h \to 0^+} \varprojlim_{R \to \infty} \mathcal{O}_{\boldsymbol{\omega}}^{h,R,\lambda}, \qquad  \mathcal{P}_{\boldsymbol{\{\omega\}}} =  \varprojlim_{\lambda \to 0^+}\varprojlim_{R \to \infty} \mathcal{P}^{R, \lambda}_{\boldsymbol{\omega}}.
$$
If $\omega(t) \rightarrow \infty$, as $t \to \infty$, Lemma \ref{compactanalrepr} implies that $\mathcal{O}_{\boldsymbol{\{\omega\}}}$ and $\mathcal{P}_{\boldsymbol{\{\omega\}}}$ are $(FS)$-spaces. The boundary value mappings
$$
\operatorname{bv}:\mathcal{O}_{\boldsymbol{(\omega)}} \rightarrow \mathcal{U}'_{\boldsymbol{(\omega)}}(\C), \qquad \operatorname{bv}:\mathcal{O}_{\boldsymbol{\{\omega\}}} \rightarrow \mathcal{A}'_{\boldsymbol{\{\omega\}}}(\R)
$$
are well defined. In the Roumieu case, this mapping is continuous. Taking Remark \ref{remarksubadd} into account, Propositions \ref{analyticreprgeneral} and  \ref{edgegeneral} yield the following corollaries.
\begin{corollary}\label{analquansub}
Let $0 < k < b < h < R$ and let $\omega$ be a weight function satisfying $(\alpha)$. 
\begin{itemize}
\item[$(i)$] If $\omega$ satisfies $(\gamma)$, then for every $f \in (\mathcal{A}^{k,\lambda}_{\boldsymbol{\omega}})'$ there is $F \in \mathcal{O}_{\boldsymbol{\omega}}^{b,R,4\lambda + 2c^{-1}}$ such that $\operatorname{bv}(F) = f$ on $\mathcal{A}^{h,4\lambda + 4c^{-1}}_{\boldsymbol{\omega}}$.
\item[$(ii)$] If $\omega$ satisfies $(\gamma)_0$, then for every $f \in (\mathcal{A}^{k,\lambda}_{\boldsymbol{\omega}})'$ there is  $F \in \mathcal{O}_{\boldsymbol{\omega}}^{b,R,4\lambda}$ such that $\operatorname{bv}(F) = f$ on $\mathcal{A}^{h,8\lambda}_{\boldsymbol{\omega}}$.
\end{itemize}
\end{corollary}
\begin{corollary}\label{edgesub}
Let $0 < b < h < R$  and let $\omega$ be a weight function satisfying $(\alpha)$.
\begin{itemize}
\item[$(i)$] If $\omega$ satisfies $(\gamma)$ and $F \in \mathcal{O}_{\boldsymbol{\omega}}^{b,R,\lambda}$ is such that $\operatorname{bv}(F) = 0$ on $\mathcal{A}^{h, \lambda + 2c^{-1}}_{\boldsymbol{\omega}}$, then $F \in \mathcal{P}_{\boldsymbol{\omega}}^{R,4\lambda + 8c^{-1}}$.
\item[$(ii)$] If $\omega$ satisfies $(\gamma)_0$ and $F \in \mathcal{O}_{\boldsymbol{\omega}}^{b,R,\lambda}$ is such that $\operatorname{bv}(F) = 0$ on $\mathcal{A}^{h, 2\lambda}_{\boldsymbol{\omega}}$, then $F \in \mathcal{P}_{\boldsymbol{\omega}}^{R,8\lambda}$.
\end{itemize}
\end{corollary}
Using exactly the same technique as in Sections \ref{sect-cohBeurling} and \ref{analyticrepresentationsroumieu} and applying Corollaries \ref{analquansub} and \ref{edgesub} (instead of Corollaries \ref{analquan} and \ref{edge}), one can show the ensuing theorem. We leave the details to the reader.   
\begin{theorem}\label{cohsub}
Let $\omega$ be a weight function satisfying $(\alpha)$. 
\begin{itemize}
\item[$(i)$] If $\omega$ satisfies $(\gamma)$, then the sequence
$$
0 \longrightarrow \mathcal{P}_{\boldsymbol{(\omega)}} \longrightarrow \mathcal{O}_{\boldsymbol{(\omega)}} \xrightarrow{\phantom,\operatorname{bv}\phantom,}  \mathcal{U}'_{\boldsymbol{(\omega)}}(\C) \longrightarrow 0
$$
is exact. 
\item[$(ii)$] If $\omega$ satisfies $(\gamma)_0$, then the sequence
$$
0 \longrightarrow \mathcal{P}_{\boldsymbol{\{\omega\}}} \longrightarrow \mathcal{O}_{\boldsymbol{\{\omega\}}} \xrightarrow{\phantom,\operatorname{bv}\phantom,}  \mathcal{A}'_{\boldsymbol{\{\omega\}}}(\R) \longrightarrow 0
$$
is topologically exact. Moreover, for every bounded set $B \subset \mathcal{A}'_{\boldsymbol{\{\omega\}}}(\R)$ there is a bounded set $A \subset \mathcal{O}_{\boldsymbol{\{\omega\}}}$ such that $\operatorname{bv}(A) = B$. 
\end{itemize}
\end{theorem}

Let us briefly discuss the notion of support in this new setting. Define $\mathcal{O}^{h,R}_{{}_\lambda \omega}(\Omega) = \mathcal{O}^{h,R, \lambda}_{\boldsymbol{\omega}}(\Omega)$, and
$$
\mathcal{O}_{\boldsymbol{(\omega)}}(\Omega+i\mathbb{R}) = \bigcup_{\lambda,h >0}\bigcap_{R >h} \mathcal{O}_{\boldsymbol{\omega}}^{h,R, \lambda}(\Omega), \qquad \mathcal{O}_{\boldsymbol{\{\omega\}}}(\Omega) = \varprojlim_{\lambda \to 0^+} \varprojlim_{h \to 0^+} \varprojlim_{R \to \infty} \mathcal{O}_{\boldsymbol{\omega}}^{h,R, \lambda}(\Omega).
$$
We suppose that $\omega$ satisfies $(\alpha)$ and $(\gamma)$ ($(\gamma)_0$). Vanishing of $f\in \mathcal{U}'_{\boldsymbol{(\omega)}}(\C)$ $(f \in \mathcal{A}'_{\boldsymbol{\{\omega\}}}(\R))$ on an open set $\Omega\subseteq \overline{\R}$ means that there is $F \in \mathcal{O}_{\boldsymbol{(\omega)}}(\Omega+i\mathbb{R})$ $(F \in \mathcal{O}_{\boldsymbol{\{\omega\}}}(\Omega))$ such that $\operatorname*{bv}(F)= f$. The definition of $\operatorname*{supp}_{\overline{\mathbb{R}}}f$ ($\operatorname*{supp}f$) should now be clear. 

We shall adopt the same kind of notations as in Section \ref{sect-support} for the rest of the spaces defined there by simply replacing $\omega$ by $\boldsymbol{\omega}$. Furthermore, all results from that section remain valid in our new context; in fact, due to Remark  \ref{remarksubadd} and the general formulation of Proposition \ref{analyticreprgeneral2}, same proofs apply here. In particular, we state the support separation theorem for future reference.
\begin{theorem}\label{splittingsub}
Let $-\infty < a \leq b < \infty$ and let $\omega$ be a weight function satisfying $(\alpha)$.
\begin{itemize}
\item[$(i)$] If $\omega$ satisfies $(\gamma)$, then the  sequence 
$$
0 \longrightarrow \mathcal{U}'_{\boldsymbol{(\omega)}}[[a,b] + i\R] \longrightarrow \mathcal{U}'_{\boldsymbol{(\omega)},a+}\oplus \mathcal{U}'_{\boldsymbol{(\omega)},b-} \xrightarrow{\phantom-\lambda\phantom-}  \mathcal{U}'_{\boldsymbol{(\omega)}}(\C) \longrightarrow 0
$$
is exact, where $\lambda((f_1,f_2)) = f_1 - f_2$.
\item[$(ii)$] If $\omega$ satisfies $(\gamma)_0$, then the sequence 
$$
0 \longrightarrow \mathcal{A}'_{\boldsymbol{\{\omega\}}}[[a,b]] \longrightarrow \mathcal{A}'_{\boldsymbol{\{\omega\}},a+} \oplus \mathcal{A}'_{\boldsymbol{\{\omega\}},b-} \xrightarrow{\phantom-\lambda\phantom-}  \mathcal{A}'_{\boldsymbol{\{\omega\}}}(\R) \longrightarrow 0
$$
is topologically exact. 
\end{itemize}
\end{theorem}
\section{Boundary values of holomorphic functions in spaces of ultradistributions of exponential type} \label{boundary ultra exponential type}\label{section laplace}
This last section is devoted to boundary values of holomorphic functions in the spaces $\mathcal{K}'_{1,\boldsymbol{(\omega)}}(\R)$ and $\mathcal{K}'_{1,\boldsymbol{\{\omega\}}}(\R)$. Our main result is a representation theorem for these two spaces as quotients of certain spaces of analytic functions (Theorem \ref{cohsubexpultradistributions}). Our arguments rely on the study of the Laplace transform of (ultra)hyperfunctions of fast growth supported in a proper interval of $\overline{\R}$ and ideas from the theory of almost analytic extensions \cite{PilipovicK, Petzsche}. 

Let us fix some notation. We shall write for complex variables 
 $\zeta = \xi + i \eta$, $z=x+iy$, and
$$
\overline{\partial} = \frac{\partial}{\partial \bar{z}} =\frac{1}{2}\left( \frac{\partial}{\partial x} + i\frac{\partial}{\partial y}  \right).
$$
The following condition for weight functions plays a role below:
\begin{itemize}
\item[$(NA)$] $\omega(t) = o(t)$.
\end{itemize}

\subsection{Laplace transforms} \label{subsection laplace} We discuss here how to define the (Fourier-)Laplace transform of ultrahyperfunctions of $\boldsymbol{(\omega)}$-type and hyperfunctions of $\boldsymbol{\{\omega\}}$-type with support in a proper closed interval $I$ of $\overline{\R}$. We assume that $\omega$ satisfies $(\alpha)$ and $(\gamma)$. 
   
 Given  $f \in \mathcal{U}'_{\boldsymbol{(\omega)}}[I + i\R]$, we define its Laplace transform as 
\begin{equation}\label{eqlaplacedef}
 \mathcal{L}\{f; \zeta\} = -\frac{1}{2\pi}\int_{\Gamma^{b}(J)}F(z)e^{iz\zeta}\dz, 
\end{equation}
for $\zeta \in \C$ in a suitable domain to be specified below and where $F \in \mathcal{O}^{h,R}_{\boldsymbol{(\omega)}}(\R \backslash I)$ is an analytic representation of $f$, that is, $\operatorname{bv}(F) = f$, $ h < b < R $, and $J$ an interval in $\overline{\R}$ such that $I \Subset J$. The definition is clearly independent of the chosen representative of $F$. In the rest of our discussion, we distinguish three cases.

\emph{Case 1: a bounded interval $I = [-a,a]$, $0 \leq a < \infty$}. In this case (\ref{eqlaplacedef}) is defined for all $\zeta\in\mathbb{C}$. In fact, $\mathcal{L}\{f; \zeta\}$ is an entire function that satisfies the estimate: there is $h > 0$ such that for every $\varepsilon > 0$ 
\begin{equation}
\sup_{\zeta \in \C} |\mathcal{L}\{f; \zeta\}|e^{-(a + \varepsilon)|\eta| - (h+\varepsilon)|\xi|} < \infty.
\label{laplace1}
\end{equation}
Conversely, it is well known \cite[Thm.\ 2.5.2]{Morimoto2} that if an entire function $G$ satisfies the bound \eqref{laplace1}, then it is the Laplace transform of an analytic functional; more precisely, there is $f  \in \mathcal{O}'[I + i[-h,h]]\subset \mathcal{U}'_{\boldsymbol{(\omega)}}[I + i\R]$ such that $G(\zeta) = \mathcal{L}\{f; \zeta\}$. If $\omega$ satisfies $(\gamma)_0$ and $f \in \mathcal{A}'_{\boldsymbol{\{\omega\}}}[I]=\mathcal{A}'[I]$, then $\mathcal{L}\{f; \zeta\}$ satisfies \eqref{laplace1} for every $h,\varepsilon > 0$, and the converse holds true: if an entire function $G$ satisfies \eqref{laplace1} for every $h, \varepsilon > 0$, then there is $f  \in \mathcal{A}'[I]$ such that $G(\zeta) = \mathcal{L}\{f; \zeta\}$. 

\emph{Case 2: a left-bounded interval $I = [-a,\infty]$, $0 \leq a<\infty$}. As was pointed out in Subsection \ref{subsection subadditive}, either $\omega$ satisfies $(NA)$ or $\omega(t) \asymp t$. 

First assume that $\omega$ satisfies $(NA)$. Then, $\mathcal{L}\{f; \zeta\}$ is analytic on the upper half-plane $\operatorname*{Im} \zeta=\eta>0$ and satisfies the bound: there are $\lambda, h > 0$ such that for every $\varepsilon > 0$ 
\begin{equation}
\sup_{\eta > 0} |\mathcal{L}\{f; \zeta\}|e^{-(a + \varepsilon)\eta - (h+\varepsilon)|\xi| -\lambda \omega^*(\eta/\lambda)} < \infty.
\label{laplace2}
\end{equation}
Moreover, the Laplace transform has as boundary value on $\mathbb{R}$ the inverse Fourier transform of $f$, namely,
$$
\lim_{\eta \to 0^+} \mathcal{L}\{f; \cdot + i \eta\} =g, \qquad \mbox{in } \mathcal{K}'_{1,\boldsymbol{(\omega)}}(\R),
$$
where $f=\widehat{g}$.
If $\omega$ satisfies $(\gamma)_0$ and $f \in \mathcal{A}'_{\boldsymbol{\{\omega\}}}[I]$, then $\mathcal{L}\{f; \zeta\}$ satisfies \eqref{laplace2} for every $\lambda, h, \varepsilon > 0$ and 
$$
\lim_{\eta \to 0^+} \mathcal{L}\{f; \cdot + i \eta\} = g, \qquad \mbox{in } \mathcal{K}'_{1,\boldsymbol{\{\omega\}}}(\R).
$$

Next, assume $\omega(t) \asymp t$, so that $\mathcal{K}_{1,\boldsymbol{(\omega)}}(\R) = \mathcal{U}_{\boldsymbol{(t)}}(\C)$ and $\mathcal{K}_{1,\boldsymbol{\{\omega\}}}(\R) = \mathcal{A}_{\boldsymbol{\{t\}}}(\R)$ are invariant under the Fourier transform. In the Beurling case there are $\lambda, h > 0$ such that $\mathcal{L}\{f; \zeta\}$ is a holomorphic function on the half-plane $\{\zeta = \xi + i\eta \, : \, \eta > \lambda\}$ and satisfies 
\begin{equation}
\sup_{\eta > \lambda} |\mathcal{L}\{f; \zeta\}|e^{-(a + \varepsilon)\eta - (h+\varepsilon)|\xi|} < \infty.
\label{laplaceBeurling}
\end{equation}
If we set
$$
G(\zeta) =
\left\{
	\begin{array}{ll}
		\mbox{$\displaystyle \mathcal{L}\{f; \zeta\}$},  &  \eta > \lambda, \\ \\
		\mbox{$0$},  & \eta < -\lambda,
	\end{array}
\right.
$$
then $\operatorname{bv}(G) = g$ in $\mathcal{U}'_{\boldsymbol{(t)}}(\C)$ (in the sense of Section \ref{subreview}), where again $f=\widehat{g}$. In the Roumieu case $\mathcal{L}\{f; \zeta\}$ is analytic on the upper half-plane and satisfies \eqref{laplaceBeurling} for every $\lambda, h,\varepsilon > 0$. Furthermore, $\operatorname{bv}(G) = g$ in $\mathcal{A}'_{\boldsymbol{\{t\}}}(\R)$ where now
$$
G(\zeta) =
\left\{
	\begin{array}{ll}
		\mbox{$\displaystyle \mathcal{L}\{f; \zeta\}$},  &  \eta > 0, \\ \\
		\mbox{$0$},  & \eta < 0.
	\end{array}
\right.
$$

\emph{Case 3: a right-bounded interval $I = [-\infty,a]$, $ 0 \leq a < \infty$.} Here the treatment is completely analogous to case $2$ but with the upper half-planes replaced by lower ones.

\subsection{Analytic representations of ultradistributions of exponential type} Our next aim is to study boundary values of holomorphic function in the spaces $\mathcal{K}'_{1,\boldsymbol{(\omega)}}(\R)$ and $\mathcal{K}'_{1,\boldsymbol{\{\omega\}}}(\R)$. For it, we need the following modified version of a construction of almost analytic extensions by Petzsche and Vogt \cite[Prop.\ 2.2]{Petzsche}.
\begin{lemma}\label{almostanalyticextensionlemma}
Let $0 < k < h$ and let $\omega$ be a weight function satisfying the conditions $(i)$-$(iv)$ of Lemma \ref{equivalentweightsub}, and let $\sigma$ be another weight function such that
\begin{equation}
\int_0^\infty te^{\omega(t) - \sigma(t)}\dt < \infty.
\label{conditionalmostanalyticextension}  
\end{equation}
Then, for every $\varphi \in \mathcal{A}_\sigma^h$  there is $\Psi \in C^\infty(\C)$ with $\Psi_{|\R} = \widehat{\varphi}$ such that
\begin{equation}
|\overline{\partial}\Psi(\zeta)| \leq \| \varphi \|^h_\omega e^{-k|\xi|}|(\omega^*)''(\eta)|e^{-\omega^*(\eta)}, \qquad \zeta\in \C \backslash \R,
\label{inequality1}
\end{equation}
and
\begin{equation} 
|\Psi(\zeta)| \leq C\| \varphi \|^h_\sigma e^{-k|\xi|}, \qquad \zeta \in \C,
\label{inequality2}
\end{equation}
where 
$$
C = 2\int_0^\infty e^{\omega(t) - \sigma(t)}\dt.
$$
\end{lemma}
\begin{proof}
The assumptions on $\omega$ imply that $\omega'$ is a smooth bijection on $(0,\infty)$. Set $H = (\omega')^{-1} \in C^\infty(0,\infty)$ and observe that
$$
\omega^*(s) = \omega(H(s)) - s H(s), \qquad s > 0.
$$
Differentiation shows that 
$$
(\omega^*)'(s) = -H(s), \qquad s > 0.
$$
We set $H(s) = H(|s|)$ for $s \in \R$, $s \neq 0$. Furthermore, since  $\omega$ is concave and increasing, we have that
\begin{equation}
t\omega'(t) \leq \omega(t), \qquad t \geq 0.
\label{concave}
\end{equation}
The rest of the proof is based on the following representation of the Fourier transform of $\varphi$ \cite[p.\ 167]{Hoskins}
$$
\widehat{\varphi}(\xi) =
\left\{
	\begin{array}{ll}
		\mbox{$\displaystyle \int_{-\infty}^\infty \varphi(x+ik)e^{-i(x+ik)\xi} \dx$},  &  \xi \leq 0, \\ \\
		\mbox{$\displaystyle \int_{-\infty}^\infty \varphi(x-ik)e^{-i(x-ik)\xi} \dx$},  & \xi \geq 0.
	\end{array}
\right.
$$
For $\zeta \in \C \backslash \R$ we define
$$
\Psi(\zeta) =
\left\{
	\begin{array}{ll}
		\mbox{$\displaystyle \int_{-H(\eta)}^{H(\eta)} \varphi(x+ik)e^{-i(x+ik)\zeta} \dx$},  &  \xi \leq 0, \\ \\
		\mbox{$\displaystyle \int_{-H(\eta)}^{H(\eta)} \varphi(x-ik)e^{-i(x-ik)\zeta} \dx$},  & \xi \geq 0,
	\end{array}
\right.
$$
and $\Psi(\xi) = \widehat{\varphi}(\xi)$ for $\xi \in \R$. Clearly, $\Psi \in C^\infty(\C\backslash \R)$. Employing \eqref{conditionalmostanalyticextension} and \eqref{concave} one can readily prove that  $\Psi \in C^1(\C)$ with ${\overline {\partial}} \Psi(\xi) = 0$, $\xi \in \R$, and thus $\Psi \in C^\infty(\C)$. We now show \eqref{inequality1}. Let $\zeta \in \C \backslash \R$ and assume $\xi \leq 0$, the case $\xi \geq 0$ can be treated similarly. Using the remarks given at the beginning of the proof, we have
\begin{align*}
|{\overline{\partial}}\Psi(\zeta)| &\leq \frac{1}{2}|H'(\eta)(\varphi(H(\eta) + ik)e^{-i(H(\eta)+ik)\zeta} + \varphi(-H(\eta) + ik)e^{-i(-H(\eta)+ik)\zeta})|  
\\
 & 
 \leq \| \varphi\|^h_\omega e^{-k|\xi|} |H'(\eta)|e^{-\omega(H(\eta))+ H(\eta)|\eta|} 
\\
& 
= \| \varphi \|^h_\omega e^{-k|\xi|}|(\omega^*)''(\eta)|e^{-\omega^*(\eta)}.
\end{align*}
It remains to establish \eqref{inequality2}. By continuity, it suffices to show this for $\zeta \in \C \backslash \R$. We  assume that $\xi \leq 0$,  the case $\xi \geq 0$ is similar. Then,
$$
|\Psi(\zeta)| \leq 2\| \varphi \|^h_\sigma e^{-k|\xi|}\int_0^{H(\eta)} e^{-\sigma(x) + x|\eta|} \dx.
$$
Notice that for $0 < x < H(\eta)$ we have $\omega'(x) > |\eta|$. Therefore, applying \eqref{concave}, we obtain that
$$
\int_0^{H(\eta)} e^{-\sigma(x) + x|\eta|} \dx \leq \int_0^{H(\eta)} e^{-\sigma(x) + x\omega'(x)} \dx \leq \int_0^\infty e^{\omega(x) - \sigma(x)}\dx,
$$
which completes the proof.
\end{proof}
\begin{corollary}\label{almostanalyticextension}
Let $\omega$ be a weight function satisfying $(\alpha)$ and $(NA)$. 
\begin{itemize}
\item[$(i)$] If $\omega$ satisfies $(\gamma)$, then for every $\psi \in \mathcal{K}_{1,\boldsymbol{(\omega)}}(\R)$ and every $\lambda, h > 0$ there is $\Psi \in C^\infty(\C)$ with $\Psi_{|\R} = \psi$ satisfying the bounds
\begin{equation}
\sup_{\zeta \in \C \backslash \R}|{\overline{\partial}}\Psi(\zeta)|e^{h|\xi|+\lambda\omega^*(\eta/\lambda)} < \infty, \qquad \sup_{\zeta \in \C}|\Psi(\zeta)|e^{h|\xi|} < \infty.
\label{inequality3}
\end{equation}
\item[$(ii)$]  If $\omega$ satisfies $(\gamma)_0$, then for every $\psi \in \mathcal{K}_{1,\boldsymbol{\{\omega\}}}(\R)$ there are $\lambda, h > 0$ and $\Psi \in C^\infty(\C)$ with $\Psi_{|\R} = \psi$ satisfying the inequalities \eqref{inequality3}.
\end{itemize}
\end{corollary}
\begin{proof}
We may assume that $\omega$ satisfies the conditions $(i)$-$(iv)$ from Lemma \ref{equivalentweightsub} (conditions $(i)$-$(iv)'$ in the Roumieu case). Petzsche and Vogt have shown \cite[Lemma 2.3]{Petzsche} that there is $\varepsilon > 0$ such that
$$
\sup_{s > 0} |(\omega^*)''(s)|e^{-\varepsilon\omega^*(s)} < \infty;
$$
if $\omega$ additionally satisfies $(iv)'$, the latter inequality holds for every $\varepsilon > 0$. Therefore the result follows by applying Lemma \ref{almostanalyticextensionlemma} to the weight ${}_\mu \omega$, for a suitable $\mu > 0$, and $\varphi = \mathcal{F}^{-1}(\psi)$.
\end{proof}

The next result gives a sufficient condition for the existence of boundary values of analytic functions in the ultradistribution spaces of exponential type $\mathcal{K}'_{1,\boldsymbol{(\omega)}}(\R)$ and $\mathcal{K}'_{1,\boldsymbol{\{\omega\}}}(\R)$.
\begin{proposition} \label{boundaryvalues}
Let $\omega$ be a weight function satisfying $(\alpha)$ and $(NA)$.
\begin{itemize}
\item[$(i)$] If $\omega$ satisfies $(\gamma)$, then every $G \in \mathcal{O}(T^R_+)$ satisfying 
\begin{equation}
\sup_{\zeta \in T^R_+}|G(\zeta)|e^{-h|\xi| - \lambda\omega^*(\eta/\lambda)} < \infty,
\label{inequality4}
\end{equation}
for some $\lambda, h > 0$, has boundary values in $\mathcal{K}'_{1,\boldsymbol{(\omega)}}(\R)$, that is, there is $g \in \mathcal{K}'_{1,\boldsymbol{(\omega)}}(\R)$, such that
$$
g = \lim_{\eta \to 0^+} G( \cdot + i \eta), \qquad \mbox{in } \mathcal{K}'_{1,\boldsymbol{(\omega)}}(\R).
$$
\item[$(ii)$]  If $\omega$ satisfies $(\gamma)_0$, then every $G \in \mathcal{O}(T^R_+)$ satisfying the inequality \eqref{inequality4} for every $\lambda, h > 0$, has boundary values in $\mathcal{K}'_{1,\boldsymbol{\{\omega\}}}(\R)$.
\end{itemize}
\end{proposition}
\begin{proof}
We only treat the Roumieu case, the Beurling case is completely analogous. Due to the Banach-Steinhaus theorem and the fact that the space $\mathcal{K}_{1,\boldsymbol{\{\omega\}}}(\R)$ is Montel, it suffices to show that 
$$
\lim_{\eta \to 0^+} \int_{-\infty}^\infty G( \xi + i \eta)\psi(\xi) \dxi
$$
exists and is finite for every $\psi \in \mathcal{K}_{1,\boldsymbol{\{\omega\}}}(\R)$. By Corollary \ref{almostanalyticextension} there is $\Psi \in C^\infty(\C)$ with $\Psi_{|\R} = \psi$ satisfying the inequalities \eqref{inequality3} for some $\lambda , h > 0$. Choose $ 0 < L < R$ and fix $ 0 < \eta <  R- L$. Applying the Stokes theorem to the rectangle $(-N,N) + i(0, L)$, $N > 0$, and the function $\widetilde{G}(\xi + iv) = G(\xi + i(\eta + v))\Psi(\xi + iv)$ we obtain that
\begin{align*}
&\int_{-N}^N G(\xi + i \eta) \psi(\xi)\dxi =  \int_{-N}^N G(\xi + i (\eta+ L)) \Psi(\xi +iL)\dxi \\
&- \int_{0}^L G(N + i (\eta+v)) \Psi(N + iv)\dv  + \int_{0}^L G(-N + i (\eta+v)) \Psi(-N + iv)\dv  \\
&+  2i\int_{-N}^N\int_0^L G(\xi + i (\eta+v)) {\overline{\partial}}\Psi(\xi+iv)\dv\dxi.
\end{align*}
The second and third integral on the right hand side tend to zero, as $N \to \infty$. Hence,
\begin{align*}
\int_{-\infty}^\infty G(\xi + i \eta) \psi(\xi)\dxi =  &\int_{-\infty}^\infty G(\xi + i (\eta+ L)) \Psi(\xi +iL)\dxi \\
&+  2i\int_{-\infty}^\infty\int_0^L G(\xi + i (\eta+v)) \overline{\partial} \Psi(\xi+iv)\dv\dxi.
\end{align*}
By Lebesgue's dominated convergence theorem, we obtain that
\begin{align*}
\lim_{\eta \to 0^+}\int_{-\infty}^\infty G(\xi + i \eta) \psi(\xi)\dxi = & \int_{-\infty}^\infty G(\xi + iL) \Psi(\xi +iL)\dxi \\
&+  2i\int_{-\infty}^\infty\int_0^L G(\xi + iv) \overline{\partial}\Psi(\xi+iv)\dv\dxi.
\end{align*}
\end{proof}
 
We now have all necessary tools to express the spaces $\mathcal{K}'_{1,\boldsymbol{(\omega)}}(\R)$ and $\mathcal{K}'_{1,\boldsymbol{\{\omega\}}}(\R)$ as quotients of spaces of analytic functions. Let $a \geq 0$ and let $\omega$ be a weight function satisfying $(\alpha)$, $(\gamma)$, and $(NA)$. We introduce the space $\mathcal{O}^{\boldsymbol{\exp},h,a}_{\boldsymbol{\omega},\lambda}(\C \backslash \R)$ consisting of all $G \in \mathcal{O}(\C \backslash \R)$ that satisfy
$$
\sup_{\zeta \in \C \backslash \R} |G(\zeta)|e^{-(a + \varepsilon)|\eta| - (h+\varepsilon)|\xi| -\lambda \omega^*(\eta/\lambda)} < \infty,
$$
for every $\varepsilon > 0$. We set
$$
\mathcal{O}^{\boldsymbol{(\exp)},a}_{\boldsymbol{(\omega)}}(\C \backslash \R) = \bigcup_{\lambda ,h > 0} \mathcal{O}^{\boldsymbol{\exp},h,a}_{\boldsymbol{\omega},\lambda}(\C \backslash \R), \qquad \mathcal{O}^{\boldsymbol{\{\exp\}},a}_{\boldsymbol{\{\omega\}}}(\C \backslash \R) =\bigcap_{\lambda ,h > 0} \mathcal{O}^{\boldsymbol{\exp},h,a}_{\boldsymbol{\omega},\lambda}(\C \backslash \R).
$$
We define the boundary value mapping as follows
$$
\operatorname{bv}: \mathcal{O}^{\boldsymbol{(\exp)},a}_{\boldsymbol{(\omega)}}(\C \backslash \R) \rightarrow \mathcal{K}'_{1,\boldsymbol{(\omega)}}(\R), \quad G \mapsto  \lim_{\eta \to 0^+} G( \cdot + i \eta) - G( \cdot - i \eta) 
$$
Proposition \ref{boundaryvalues} guarantees  that $\operatorname{bv}$ is well defined. Moreover, if $\omega$ satisfies $(\gamma)_0$, then $\operatorname{bv}(\mathcal{O}^{\boldsymbol{\{\exp\}},a}_{\boldsymbol{\{\omega\}}}(\C \backslash \R)) \subseteq  \mathcal{K}'_{1,\boldsymbol{\{\omega\}}}(\R)$.
We also write $\mathcal{O}^{\boldsymbol{\exp},h,a}(\C)$ for the space of all entire functions $G \in \mathcal{O}(\C)$ that satisfy
$$
\sup_{\zeta \in \C} |G(\zeta)|e^{-(a + \varepsilon)|\eta| - (h+\varepsilon)|\xi|} < \infty,
$$
for every $\varepsilon > 0$ and set
$$
\mathcal{O}^{\boldsymbol{(\exp)},a}(\C) = \bigcup_{h > 0} \mathcal{O}^{\boldsymbol{\exp},h,a}(\C), \qquad \mathcal{O}^{\boldsymbol{\{\exp\}},a}(\C) =\bigcap_{h > 0} \mathcal{O}^{\boldsymbol{\exp},h,a}(\C).
$$
\begin{theorem}\label{cohsubexpultradistributions}
Let $a \geq 0$ and let $\omega$ be a weight function satisfying $(\alpha)$ and $(NA)$. 
\begin{itemize}
\item[$(i)$] If $\omega$ satisfies $(\gamma)$, then the sequence
$$
0 \longrightarrow \mathcal{O}^{\boldsymbol{(\exp)},a}(\C) \longrightarrow \mathcal{O}^{\boldsymbol{(\exp)},a}_{\boldsymbol{(\omega)}}(\C \backslash \R) \xrightarrow{\phantom,\operatorname{bv}\phantom,}  \mathcal{K}'_{1,\boldsymbol{(\omega)}}(\R) \longrightarrow 0
$$
is exact. 
\item[$(ii)$] If $\omega$ satisfies $(\gamma)_0$, then the sequence
$$
0 \longrightarrow \mathcal{O}^{\boldsymbol{\{\exp\}},a}(\C) \longrightarrow \mathcal{O}^{\boldsymbol{\{\exp\}},a}_{\boldsymbol{\{\omega\}}}(\C \backslash \R) \xrightarrow{\phantom,\operatorname{bv}\phantom,}  \mathcal{K}'_{1,\boldsymbol{\{\omega\}}}(\R) \longrightarrow 0
$$
is exact.
\end{itemize}
\end{theorem}
\begin{proof}
We only give the proof in the Roumieu case, the Beurling case is similar. The fact that $\ker\operatorname{bv} = \mathcal{O}^{\boldsymbol{\{\exp\}},a}(\C)$ follows from Theorem \ref{Roumieucoh} (with $\omega(t) = t$). So, we only need to show that the boundary value mapping is surjective. Let $g \in \mathcal{K}'_{1,\boldsymbol{\{\omega\}}}(\R)$ and set $f = \widehat{g} \in \mathcal{A}'_{\boldsymbol{\{\omega\}}}(\R)$. By Theorem \ref{splittingsub} there are $f_+ \in \mathcal{A}'_{\boldsymbol{\{\omega\}},(-a)+}$ and $f_- \in \mathcal{A}'_{\boldsymbol{\{\omega\}},a-}$ such that $f = f_+ - f_-$. Define
$$
G(\zeta) =
\left\{
	\begin{array}{ll}
		\mbox{$\displaystyle \mathcal{L}\{f_+; \zeta\}$},  &  \eta > 0, \\ \\
		\mbox{$\displaystyle \mathcal{L}\{f_-; \zeta\}$},  &  \eta < 0.
	\end{array}
\right.
$$
From the discussion in Subsection \ref{subsection laplace} on the Laplace transform, it is clear that $G \in \mathcal{O}^{\boldsymbol{\{\exp\}},a}_{\boldsymbol{\{\omega\}}}(\C \backslash \R)$ and $\operatorname{bv}(G) = g$.
\end{proof}
As an application of Theorem \ref{cohsubexpultradistributions}, we characterize in a precise fashion those analytic functions on the upper half-plane that are the Laplace transform of an (ultra)hyperfunction of $\boldsymbol{\omega}$-type supported on a fixed half-axis. The following result is a theorem of Paley-Wiener type.
\begin{theorem}\label{characterizationlaplace}
Let $a \geq 0$, let $\omega$ be a weight function satisfying $(\alpha)$ and $(NA)$, and suppose that $G$ is an analytic function on the upper half-plane.
\begin{itemize}
\item[$(i)$]  If $\omega$ satisfies $(\gamma)$, then $G$ satisfies the estimate
\begin{equation}\label{eqboundLaplace}
\sup_{\eta > 0} |G(\zeta)|e^{-(a + \varepsilon)\eta - (h+\varepsilon)|\xi| -\lambda \omega^*(\eta/\lambda)} < \infty,
\end{equation}
for some $h, \lambda > 0$ and
 for every $\varepsilon > 0$ if and only if there is $f \in \mathcal{U}'_{{\boldsymbol{(\omega)}},(-a)+}$ such that $G(\zeta) = \mathcal{L}\{f;\zeta\}$.
\item[$(ii)$] If $\omega$ satisfies $(\gamma)_0$, then $G$ satisfies \eqref{eqboundLaplace} for every $h,\lambda, \varepsilon > 0$ if and only if there is $f \in \mathcal{A}'_{{\boldsymbol{\{\omega\}}},(-a)+}$ such that $G(\zeta) = \mathcal{L}\{f;\zeta\}$.
\end{itemize}
\end{theorem}
\begin{proof}
We only treat the Roumieu case, the Beurling case is analogous. It has already been pointed out in Subsection \ref{subsection laplace} that the Laplace transform of an element of $\mathcal{A}_{{\boldsymbol{\{\omega\}}},(-a)+}$ satisfies the required bounds. Conversely, let $G$ be an analytic function on the upper half-plane satisfying \eqref{eqboundLaplace} for every $h,\lambda, \varepsilon > 0$. By Proposition \ref{boundaryvalues} there is $g \in \mathcal{K}'_{1,\boldsymbol{\{\omega\}}}(\R)$ such that
$$
g = \lim_{\eta \to 0^+} G( \cdot + i \eta), \qquad \mbox{in } \mathcal{K}'_{1,\boldsymbol{\{\omega\}}}(\R).
$$
Let $f = \widehat{g} \in \mathcal{A}'_{\boldsymbol{\{\omega\}}}(\R)$. By Theorem \ref{splittingsub}, there are $f_+ \in \mathcal{A}'_{\boldsymbol{\{\omega\}},(-a)+}$ and $f_- \in \mathcal{A}'_{\boldsymbol{\{\omega\}},a-}$ such that $f = f_+ - f_-$. Define
$$
\widetilde{G}(\zeta) =
\left\{
	\begin{array}{ll}
		\mbox{$\displaystyle \mathcal{L}\{f_+; \zeta\}$},  &  \eta > 0, \\ \\
		\mbox{$\displaystyle \mathcal{L}\{f_-; \zeta\}$},  &  \eta < 0.
	\end{array}
\right.
$$
Notice that $\widetilde{G} \in \mathcal{O}^{\boldsymbol{\{\exp\}},a}_{\boldsymbol{\{\omega\}}}(\C \backslash \R)$ and $\operatorname{bv}(\widetilde{G}) = g$. Hence, by Theorem \ref{cohsubexpultradistributions}, there is $H \in 
\mathcal{O}^{\boldsymbol{\{\exp\}},a}(\C)$ such that $G = \widetilde{G} + H$ on the upper half-plane. Since there is $h \in \mathcal{A}'[[-a,a]]= \mathcal{A}_{\boldsymbol{\{\omega\}}}'[[-a,a]]$ such that $H(\zeta) = \mathcal{L}\{h;\zeta\}$ (Case $1$ in Subsection \ref{subsection laplace}), we conclude that 
$G(\zeta) = \mathcal{L}\{f_+ + h;\zeta\}$.
\end{proof}

 If $\omega$ satisfies $(\alpha)$ but not $(NA)$, we must have $\omega(t) \asymp t$ and thus $\mathcal{K}_{1,\boldsymbol{(\omega)}}(\R) = \mathcal{U}_{\boldsymbol{(t)}}(\C)$ and $\mathcal{K}_{1,\boldsymbol{\{\omega\}}}(\R) = \mathcal{A}_{\boldsymbol{\{t\}}}(\R)$. In this case, the counterparts of Theorems \ref{cohsubexpultradistributions} and \ref{characterizationlaplace} go back to the work of Silva and Morimoto \cite{Morimoto-78,Silva}. For the sake of completeness, we end this article by stating these theorems. Let $a \geq 0$. We define $\mathcal{O}^{\boldsymbol{\exp},h,a}(\C \backslash \overline{T^\lambda})$ as the space of all $G \in \mathcal{O}(\C \backslash \overline{T^\lambda})$ such that
$$
\sup_{\zeta \in \C \backslash \overline{T^\lambda}} |G(\zeta)|e^{-(a + \varepsilon)|\eta| - (h+\varepsilon)|\xi|}< \infty,
$$
for every $\varepsilon > 0$, and 
$$
\mathcal{O}^{\boldsymbol{(\exp)},a} = \bigcup_{\lambda ,h > 0} \mathcal{O}^{\boldsymbol{\exp},h,a}(\C \backslash \overline{T^\lambda}), \qquad \mathcal{O}^{\boldsymbol{\{\exp\}},a}(\C \backslash \R) =\bigcap_{\lambda ,h > 0} \mathcal{O}^{\exp,h,a}(\C \backslash \overline{T^\lambda}).
$$
The proofs of the ensuing two results go along the same lines as those of Theorems \ref{cohsubexpultradistributions} and \ref{characterizationlaplace}, we therefore choose to omit them.
\begin{theorem}
Let $a\geq 0$. The sequences
$$
0 \longrightarrow \mathcal{O}^{\boldsymbol{(\exp)},a}(\C) \longrightarrow \mathcal{O}^{\boldsymbol{(\exp)},a} \xrightarrow{\phantom,\operatorname{bv}\phantom,}  \mathcal{U}'_{\boldsymbol{(t)}}(\C) \longrightarrow 0
$$
and
$$
0 \longrightarrow \mathcal{O}^{\boldsymbol{\{\exp\}},a}(\C) \longrightarrow \mathcal{O}^{\boldsymbol{\{\exp\}},a}(\C \backslash \R) \xrightarrow{\phantom,\operatorname{bv}\phantom,}  \mathcal{A}'_{\boldsymbol{\{t\}}}(\R) \longrightarrow 0
$$
are exact (the boundary value operator being interpreted in the sense of Section \ref{subreview}).
\end{theorem}
\begin{theorem}
Let $a \geq 0$.
\begin{itemize}
\item[$(i)$]  Suppose $G$ is analytic on the half-plane $\{\zeta = \xi + i\eta \in \C \, : \, \eta > \lambda\}$ for some $\lambda > 0$, and satisfies 
\begin{equation}
\sup_{\eta > \lambda} |G(\zeta)|e^{-(a + \varepsilon)|\eta| - (h+\varepsilon)|\xi|}< \infty,
\label{charlaplace2}
\end{equation}
for some $h > 0$ and every $\varepsilon > 0$, then there is $f \in \mathcal{U}'_{{\boldsymbol{(t)}},(-a)+}$ with $G(\zeta) = \mathcal{L}\{f;\zeta\}$. Conversely, the Laplace transform $\mathcal{L}\{f;\zeta\}$ of any $f \in \mathcal{U}'_{{\boldsymbol{(t)}},(-a)+}$ is analytic on some half-plane $\{\zeta = \xi + i\eta \in \C \, : \, \eta > \lambda\}$, $\lambda > 0$, and satisfies \eqref{charlaplace2} for some $h > 0$ and every $\varepsilon > 0$.
\item[$(ii)$] A function $G$ analytic on the upper half-plane satisfies \eqref{charlaplace2} for every $h,\lambda, \varepsilon > 0$ if and only if there is $f \in \mathcal{A}'_{{\boldsymbol{\{t\}}},(-a)+}$ such that $G(\zeta) = \mathcal{L}\{f;\zeta\}$.
\end{itemize}
\end{theorem}

\end{document}